\documentclass[12pt,oneside,reqno]{amsart}

\usepackage[utf8]{inputenc}
\usepackage[T1]{fontenc}
\usepackage[russian,ngerman,english]{babel}
\usepackage{lmodern}
\usepackage{hyperref}
\usepackage{amsmath}
\usepackage{amsthm}
\usepackage{amsfonts}
\usepackage{amssymb}
\usepackage{amscd}
\usepackage{amsbsy}
\usepackage{subcaption}
\usepackage{pdfpages}
\usepackage{fancyhdr}
\usepackage[hmargin=25mm,vmargin=25mm,bmargin=30mm,tmargin=25mm]{geometry}
\usepackage{geometry}
\usepackage{setspace}
\usepackage{xcolor}
\usepackage{pifont}
\usepackage{graphicx}
\usepackage{tabularx}
\usepackage{epsfig}
\usepackage[all]{xy}
\usepackage{tikz}
\usetikzlibrary{matrix}

  \usepackage{fixmath}

  \usepackage{appendix}

  \usepackage{enumerate}

  \usepackage[sort, numbers]{natbib}
  \usepackage{bibentry}
  \newtheorem{theorem}{Theorem}[section]

    \newtheorem*{theorem*}{Theorem}
\newtheorem*{lemma*}{Lemma}

  \newtheorem{theoremm}{Theorem}
  \newtheorem{lemma}{Lemma}

  \newtheorem{corollary}{Corollary}
  
  \theoremstyle{definition}
  \newtheorem{definition}{Definition}

  \newtheorem*{remark}{Remark}

  \newtheorem{conjecture}{Conjecture}
  \numberwithin{equation}{section}

  \newcommand{\N}{{\mathbb N}}

  \newcommand{\R}{{\mathbb R}}
  \renewcommand{\C}{{\mathbb C}}
  \newcommand{\D}{{\mathbb D}}

  \newcommand{\E}{{\mathsf E}}
  \newcommand{\Sm}{{\mathsf S}}

  \newcommand{\cW}{{\mathcal{W}}}

   \newcommand{\eps}{{\epsilon}}

  \renewcommand{\a}{\alpha}

  \renewcommand{\Cap}{\operatorname{Cap}}

  \renewcommand{\i}{\infty}

\usepackage{caption}
\DeclareCaptionLabelFormat{cont}{#1~#2\alph{ContinuedFloat}}
  \newcommand{\bbD}{\mathbb{D}}
  \newcommand{\bbN}{\mathbb{N}}

  \newcommand{\bbR}{\mathbb{R}}
  \newcommand{\bbC}{\mathbb{C}}

  \allowdisplaybreaks

\author[J.~S.~Christiansen, B.~Eichinger and O.~Rubin]{Jacob S. Christiansen$^{1,4}$, Benjamin Eichinger$^{2,5}$, and Olof Rubin$^{3}$}

\thanks{$^1$ Centre for Mathematical Sciences, Lund University, Box 118, 22100 Lund, Sweden.
E-mail: jacob$\_$stordal.christiansen@math.lth.se}

\thanks{$^2$ Institute for Analysis and Scientific Computing, Vienna University of Technology, Wien, A-1040, Austria.
E-mail: benjamin.eichinger@tuwien.ac.at}

\thanks{$^3$ Centre for Mathematical Sciences, Lund University, Box 118, 22100 Lund, Sweden.
E-mail: olof.rubin@math.lth.se}

\thanks{$^4$ Research supported by VR grant 2018-03500 from the Swedish Research Council and in part by DFF research project 1026-00267B from the Independent Research Fund Denmark.}

\thanks{$^5$ Research supported by the Austrian Science Fund FWF, Project No. P33885}

\title{Extremal polynomials and polynomial preimages}
\begin{document}
	\maketitle
	
\begin{abstract}	
	This article examines the asymptotic behavior of the Widom factors, denoted $\cW_n$, for Chebyshev polynomials of finite unions of Jordan arcs.  
	We prove that, in contrast to Widom's proposal in \cite{Widom1969}, when dealing with a single smooth Jordan arc, $\cW_n$ converges to 2 exclusively when the arc is a straight line segment.
	Our main focus is on analysing polynomial preimages of the interval $[-2,2]$, and we provide a complete description of the asymptotic behavior of $\cW_n$ for symmetric star graphs and quadratic preimages of $[-2,2]$.  	
	 We observe that in the case of star graphs, the Chebyshev polynomials and the polynomials orthogonal with respect to equilibrium measure share the same norm asymptotics, suggesting a potential extension of the conjecture posed in \cite{CSZ-review}.
	 %by Alpan and Zinchenko \cite{AlpanZinchenko2020}. 
	 Lastly, we propose a possible connection between the $S$-property and Widom factors converging to $2$.
	
%	The Chebyshev polynomial of degree $n$ corresponding to a compact set $\E\subset \C$ is the monic polynomial, denoted $T_n^\E$, of degree $n$ which minimises the supremum norm on $\E$. It is known that the quantity $\|T_n^{\E}\|_\E/\Cap(\E)^n$ is bounded and has a limit as $n\rightarrow \infty$ for a wide variety of  sets $\E$. For real sets $\E$ this quantity is lower bounded by $2$. Here we investigate sets in the complex plane where the limit is $2$ and relate this convergence with a geometric property of the sets.	
	\end{abstract}

\medskip	
	
	{\bf Keywords} {Chebyshev polynomials, Widom factors, polynomial preimages, star graphs}
	
\medskip	
	
	{\bf Mathematics Subject Classification} {41A50, 30C10, 26D05}

	\section{Introduction}
	
	Let $\E$ be a compact subset of $\C$ with at least $n+1$ points. The $n$th Chebyshev polynomial %of degree $n$ corresponding to 
	of $\E$, denoted $T_n^\E$, is the unique monic polynomial of degree $n$ which minimises the supremum norm on $\E$. In other words, $T_n^\E$ is the polynomial
	\[
		T_{n}^\E(z) = z^{n}+\sum_{k=0}^{n-1}a_kz^k
	\]
	which satisfies
	\[ \|T_n^\E\|_\E :=\max_{z\in \E} |T_n^{\E}(z)|=
	 \min_{b_0, b_1, \ldots, b_{n-1} \in \C}
	 \max_{z\in \E} \left|z^{n}+\sum_{k=0}^{n-1} b_k z^k\right|.
	\]
	
	Facts regarding existence and uniqueness of $T_n^\E$ can be found in, e.g., \cite{SmirnovLebedev, Lorentz66}. See also \cite{CSZ-I, CSZ-IV} for a recent account on the basic theory of Chebyshev polynomials. 
	These polynomials were initially studied by Chebyshev \cite{Chebyshev1854, Chebyshev1859} in the case where $\E=[-1,1]$. In this situation, he showed that the polynomials are explicitly given by the formula
	\begin{equation}
	\label{Tn}	
		T_n(x):=T_n^{[-1,1]}(x) = 2^{1-n}\cos\bigl(n\arccos(x)\bigr),\quad |x|\leq 1.
	\end{equation}
	This representation further hints at a property of the Chebyshev polynomials that holds for arbitrary compact subsets of the real line. %If $\E\subset\R$ is a compact set containing at least $n+1$ points, then 
	A monic degree $n$ polynomial, $P_n$, is the Chebyshev polynomial of $\E\subset\R$ if and only if there exists points $x_0<x_1<\ldots<x_n$ in $\E$ such that
		\[
			P_n(x_k)=(-1)^{n-k}\|P_n\|_\E,\quad k=0, 1, \ldots, n.
		\]
	This characterising property of the Chebyshev polynomials is called \emph{alternation} and %is a characterising property of the polynomials. Using the alternation property, 
	one can use it to prove several facts concerning their asymptotic behaviour % of the Chebyshev polynomials corresponding to real subsets 
	(see, e.g., \cite{CSZ-I, CSZ-II}). 
	
	It should be noted that alternation fails to hold for Chebyshev polynomials of non-real compact subsets of $\C$. Instead, the asymptotics of such polynomials is typically studied using potential theoretic methods. This is an approach dating back to Faber \cite{Faber1920}, Fekete \cite{Fekete1923}, and Szeg\H{o} \cite{Szego1924}. Faber investigated Chebyshev polynomials by constructing trial polynomials, the so-called Faber polynomials, from the associated conformal map which maps the exterior of $\E$ to the exterior of the closed unit disk, $\overline{\D}$. Of course, the existence of this conformal map assumes that the complement of $\E$ is simply connected on the Riemann sphere.
	
As for notation, let $\overline{\C}$ denote the Riemann sphere $\C\cup\{\infty\}$.	Throughout the paper, we shall make use of the following important notions
	\begin{itemize}
		\item $\Cap(\E)$, the logarithmic capacity of $\E$ \vspace{0.1cm}
		\item $G_\E(\,\cdot\,):=G_{\E}(\,\cdot\,, \infty)$, the Green's function for $\overline{\C}\setminus \E$ with pole at $\infty$  \vspace{0.1cm}
		\item $\mu_\E$, the equilibrium measure of $\E$
	\end{itemize}
	Recall also the relation
	\begin{equation}
	\label{rel}
	   \int \log\vert x-z\vert \,d\mu_\E(x)=G_\E(z)+\log\Cap(\E).
	\end{equation}
	%When $\zeta = \infty$, we write $G_\E(\, \cdot \, ,\infty) = G_\E(\, \cdot \, )$ for the sake of simplicity.
	For a more in depth account of potential theory, we refer the reader to, e.g., \cite{ArmGar01, garnett_marshall_2005, Hel09, Lan72, ransford_1995}.

	\subsection{Widom factors for $T_n^\E$} % the $L^\infty$ problem}

	One way of quantatively describing the norm of the Chebyshev polynomials using logarithmic capacity is via the Faber--Fekete--Szeg\H{o} theorem which states that
	\begin{equation}
		\lim_{n\rightarrow \infty}\|T_n^{\E}\|_\E^{1/n} = \Cap(\E)
		\label{eq:FFS}
	\end{equation}
	for any compact set $\E\subset\C$ (see, e.g., \cite[Chapter 5.5]{ransford_1995}). This implies that $\Cap(\E)^n$ is the leading order behaviour of $\|T_n^{\E}\|_{\E}$, and \eqref{eq:FFS} limits the way that the so-called Widom factors defined by
	\begin{equation}
		\cW_{n,\infty}(\E) := \frac{\|T_n^\E\|_\E}{\Cap(\E)^n}
		\label{eq:widom_factor_def}
	\end{equation} can grow as $n$ increases. For a wide variety of sets, this quantity is known to be bounded in $n$ and its asymptotic behaviour is of particular interest. See, e.g., \cite{And16, And17, CSZ-IV} for more details.
	
Another classical result in the theory of Chebyshev polynomials is	
	\begin{theorem*}
		Let $\E\subset \C$ be a compact set. Then
		\begin{equation}
			\|T_n^\E\|_\E\geq \Cap(\E)^n
			\label{eq:szego_inequality}
		\end{equation}
		and if $\E\subset \R$, we even have 
		\begin{equation}
			\|T_n^\E\|_\E\geq 2\Cap(\E)^n.
			\label{eq:schiefermayr_inequality}
		\end{equation}
	\end{theorem*}
	While \eqref{eq:szego_inequality} goes back to Szeg\H{o} \cite{Szego1924}, the inequality in \eqref{eq:schiefermayr_inequality} is more recent and due to Schiefermayr \cite{Schiefermayr2008}. We shall present a proof of these statements below. Partly because our method will be used in later parts of the paper and is much shorter than the one presented in \cite{Schiefermayr2008}, and partly for completeness. Our proof rests on the following formula for capacity of polynomial preimages. 
	\begin{lemma*} 
	Let $\E$ be a compact subset of $\C$ and suppose $P(z) = \sum_{k=0}^{m}a_kz^k$ is a polynomial with $a_m\neq 0$. If
		\[
			\E_P := \bigl\{z:P(z)\in \E\bigr\} = P^{-1}(\E),
		\]
		then
		\begin{equation}
			G_{\E_P}(z) = \frac{1}{m}G_\E\bigl(P(z)\bigr)\quad \text{and}\quad \Cap(\E_P) = \left(\frac{\Cap(\E)}{|a_m|}\right)^{1/m}.	\label{eq:green_capacity_polynomial_preimage}
		\end{equation}
	\end{lemma*}	
For a proof of this fact, see \cite[Theorem 5.2.5]{ransford_1995}.
	\begin{proof}[Proof of theorem]
		Let $\E\subset \C$ be an arbitrary compact set. Clearly,
		\[
			\E\subset (T_n^\E)^{-1}\Big(\bigl\{z:|z|\leq \|T_n^{\E}\|_{\E}\bigr\}\Big).
		\]
		Applying \eqref{eq:green_capacity_polynomial_preimage} together with the facts that $\Cap$ is monotone with respect to set inclusion and a disk of radius $r>0$ has capacity equal to $r$, we obtain that
		\[\Cap(\E)\leq \Cap\Big(\bigl\{z:|z|\leq \|T_n^{\E}\|_{\E}\bigr\}\Big)^{1/n} = \|T_n^{\E}\|^{1/n}.\]
		This proves \eqref{eq:szego_inequality}.
		
		If $\E\subset \R$, then $T_n^{\E}$ is easily shown to have only real coefficients. Hence
		\[\E\subset (T_n^\E)^{-1}\Big(\bigl\{z:z\in \bigl[-\|T_n^\E\|_\E,\|T_n^{\E}\|_\E\bigr]\bigr\}\Big).\]
		As the capacity of an interval $[a,b]$ equals ${(b-a)}/{4}$, yet another application of \eqref{eq:green_capacity_polynomial_preimage} implies that
		\[
			\Cap(\E)\leq \Cap\Big(\bigl\{z:z\in \bigl[-\|T_n^\E\|_\E,\|T_n^{\E}\|_\E\bigr]\bigr\}\Big)^{1/n} = \left(\frac{\|T_n^\E\|_\E}{2}\right)^{1/n},
		\]
		proving \eqref{eq:schiefermayr_inequality}.
	\end{proof}

\begin{remark}	
Note that the above theorem %\eqref{eq:szego_inequality} and \eqref{eq:schiefermayr_inequality} 
can be restated in terms of the Widom factors as
	\begin{equation}
		\cW_{n,\infty}(\E)\geq 1 \; \text{ for } \; \E\subset \C 
	\end{equation}
and	
         \begin{equation}	
%		\quad \text{and} \quad
		\cW_{n,\infty}(\E)\geq 2 \; \text{ when } \; \E\subset \R.
	\end{equation}
\end{remark}	
	The sets for which the Szeg\H{o} lower bound \eqref{eq:szego_inequality} is saturated for \emph{some} value of $n$ were determined in \cite{CSZ-III}. With $O\partial(\,\cdot\,)$ denoting the outer boundary and $\partial\D$ the unit circle, the authors proved that
	\[\cW_{n,\infty}(\E) = 1 \quad \text{if and only if } \quad O\partial(\E) =P^{-1}(\partial \bbD) \]
	for some polynomial $P$ of degree $n$. Regarding Schiefermayr's lower bound, it was proven in \cite{Totik2011} (see also \cite{CSZ-III}) that for $\E\subset \R$, 
	\begin{equation}
	\cW_{n,\infty}(\E) = 2\quad \text{if and only if }\quad \E = P^{-1}\bigl([-2,2]\bigr)     
	\label{eq:widom_possible_limits}
	\end{equation}
	for some degree $n$ polynomial $P$.
		
	Now the question remains in which cases we have an asymptotic saturation of these lower bounds. More precisely, when does it happen that
	\[\lim_{n\rightarrow \infty}\cW_{n,\infty}(\E) = 1 \text{ or }2?\]
	It is known that if $\E$ is the closure of a Jordan domain with boundary curve of class $C^{2+\epsilon}$ (i.e., its coordinates are $C^{2+\epsilon}$ functions\footnote{A function of a real variable belongs to $C^{2+\epsilon}$ if its $2$nd derivative satisfies a Lipschitz condition with some positive exponent.} of arc length), then %for some $\epsilon>0$
	\begin{equation}
		\lim_{n\rightarrow \infty}\cW_{n,\infty}(\E)=1.
		\label{eq:asymptotic_szego_lower_bound}
	\end{equation}
	This was first shown in the case where $\E$ has analytic boundary by Faber \cite{Faber1920} and then extended to the case where the boundary curve is of class $C^{2+\epsilon}$ by Widom \cite{Widom1969}. %With $C^{2+\epsilon}$ we mean a function which is twice continuously differentiable and such that the second derivative satisfies a H\"{o}lder condition with exponent $\epsilon$. 
	However, there may well be many other connected sets $\E\subset \C$ for which \eqref{eq:asymptotic_szego_lower_bound} holds. 

	For $\E\subset \R$, Totik \cite{Totik2014} completely characterised the sets which asymptotically saturate the Schiefermayr lower bound \eqref{eq:schiefermayr_inequality}. He proved that 
	\begin{equation}
		\lim_{n\rightarrow \infty}\cW_{n,\infty}(\E)= 2
		\label{eq:asymptotic_schiefermayr}
	\end{equation}
	if and only if $\E$ is an interval, in which case $\cW_{n,\infty}(\E) = 2$ for every $n$.

	One may ask --- and this is a main point of the present article --- if there are more subsets $\E\subset \C$ for which \eqref{eq:asymptotic_schiefermayr} holds true. Widom \cite{Widom1969} conjectured that any sufficiently nice set which contains an arc component should satisfy \eqref{eq:asymptotic_schiefermayr}. However, this was shown to be false even in the case of Jordan arcs. In particular, Thiran and Detaille \cite{ThiranDetaille1991} observed that if $\E_\alpha = \{z: |z|=1,\, |\arg z|\leq \alpha\}$ with $\alpha\in (0,\pi)$, then
	\begin{equation}
		\lim_{n\rightarrow \infty}\cW_{n,\infty}(\E_\alpha)= 2\cos^2\left(\alpha/4\right)<2.
		\label{eq:thiran_detaille}
	\end{equation}	
	See also \cite{Schiefermayr2019} and \cite{Eichinger2017}.
		
	The aim of this article is to study sets $\E\subset \C$ which satisfy \eqref{eq:asymptotic_schiefermayr} and investigate what properties may lie behind. 
	First of all, by combining results of Stahl \cite{Stahl2012} and Alpan \cite{Alpan2022}, we are able to prove the following result which essentially is a reformulation of \cite[Theorem 1.3]{Alpan2022}.
	\begin{theoremm}
		Let $\E\subset\C$ be a Jordan arc of class $C^{2+\epsilon}$. Then
		\begin{equation}  \label{Alpan}
		   \limsup_{n\rightarrow \infty}\cW_{n,\infty}(\E)\leq 2
		\end{equation}
		and equality holds if and only if $\E$ is a straight line segment.
		\label{thm:widom_factor_arc}
	\end{theoremm}
	In \cite{Alpan2022}, Alpan showed that \eqref{Alpan} is always satisfied and gave a condition for when equality holds true in terms of the boundary behaviour of the Green's function. He also drew the conclusion that the inequality is strict if the arc fails to be analytic. Our contribution consists of showing that the only smooth arc for which equality holds is an interval. We show this using the $S$-property which was introduced and studied in detail by Stahl (see, e.g., \cite{Stahl85-I-II, Stahl85-III, Stahl2012, Martinez-Finkelshtein2011} and Definition \ref{def:s-property} below). To be more precise, our proof hinges on the connection  between the $S$-property and what are known as Chebotarev sets. %which are connected sets of minimal capacity containing a given collection of points. %By combining the results of \cite{Stahl2012} and \cite{Schiefermayr:2014}, it can furthermore be shown that if $P$ is a polynomial and $\E_P = P^{-1}([-2,2])$ is connected, then $\E_P$ satisfies the $S$-property.
	
	In this paper we shall exhibit several examples of sets $\E$ satisfying \eqref{eq:asymptotic_schiefermayr} and one of the properties that these sets have in common, apart from being polynomial preimages of an interval, is the fact that they all satisfy the $S$-property. This suggests that the $S$-property could be of importance for a set $\E$ to satisfy \eqref{eq:asymptotic_schiefermayr}. 
	
	Our studies were initiated by the following family of examples
	\begin{theoremm} \label{thm:widom-factors}
		For $m\in \N$, let 
		\begin{equation}
		\label{Em}
		 \E_m = \bigl\{z: z^m\in [-2,2]\bigr\}.
		\end{equation}  
		Then 
		\begin{equation}
		\cW_{mn,\infty}(\E_m)=2, \quad n\geq 1 
		\end{equation}
		and \eqref{eq:asymptotic_schiefermayr} holds true, that is, 
		\begin{equation} \label{limit n}
		 \lim_{n\rightarrow \infty}\cW_{n,\infty}(\E_m)= 2.
		\end{equation} 
		Moreover, 
		\begin{equation} \label{limit m}
		 \lim_{m\to\infty} \cW_{2m-1,\infty}(\E_m)= 4.
		\end{equation}
	\end{theoremm}
	The latter part of the theorem shows that the limit in \eqref{limit n} by no means is uniform in $m$. It also exemplifies that no matter how large we take $n\in\N$, there always exist a compact connected set $\E\subset\C$ and $m\geq n$ so that $\cW_{m,\infty}(\E)>4-\epsilon$.	 
		%The sets $\E_m$ posses certain extremal properties among symmetric sets as we will show in Corollary \ref{cor:msymmetric}. Namely if a compact connected set, $\E$, satisfies $e^{2\pi i/m}\E = \E$ then for $0\leq l \leq m-1$, $\cW_{l}(\E)\leq 4^{l/m}$ with equality if and only if $\E$ is a star set.
	\begin{figure}[h!]
	%\ContinuedFloat*
			\centering
			\begin{minipage}{.33\textwidth}
		  		\centering
  				\includegraphics[width=.9\linewidth]{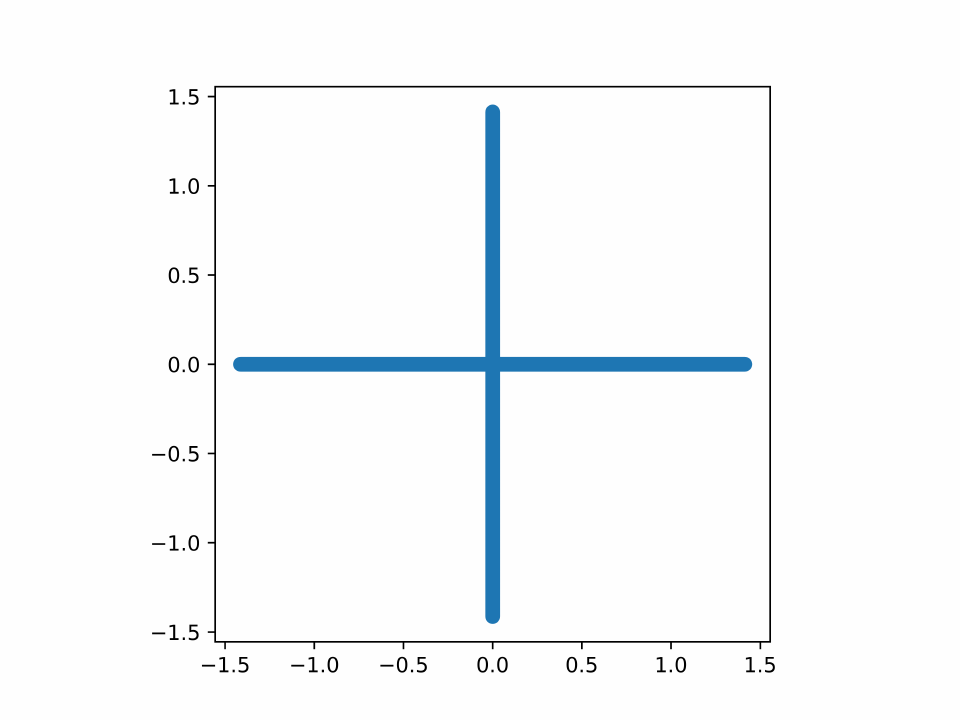}
				\caption{$\E_2$}
				\label{fig-E2}
			\end{minipage}%
			%\ContinuedFloat
			\begin{minipage}{.33\textwidth}
				\centering
				\includegraphics[width=.9\linewidth]{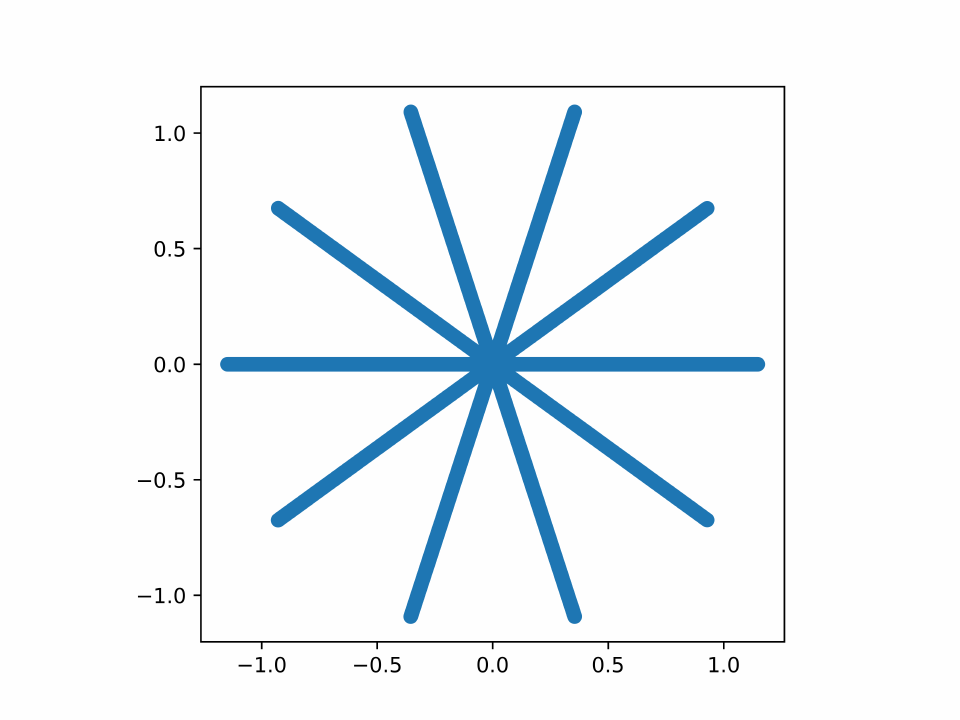}
				\caption{$\E_{5}$}
				\label{fig-E5}
			\end{minipage}
			%\ContinuedFloat
			\begin{minipage}{.33\textwidth}
		  		\centering
  				\includegraphics[width=.9\linewidth]{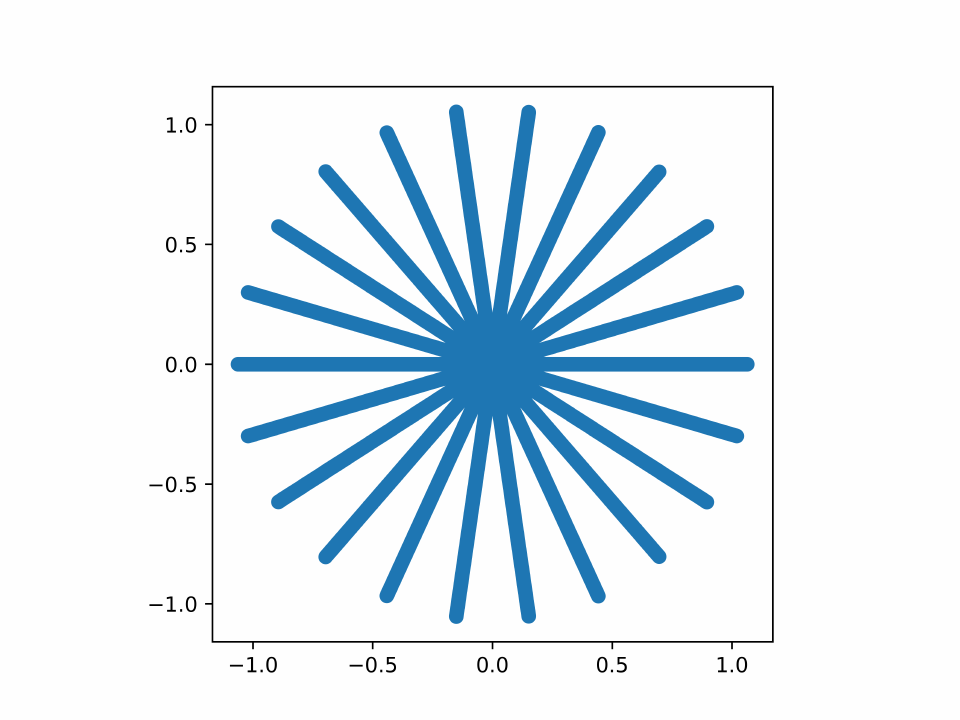}
				\caption{$\E_{11}$}
				\label{fig-E11}
			\end{minipage}%
		\end{figure}

	Extremal polynomials of the sets $\E_m$ were previously studied by Peherstorfer and Steinbauer \cite{PeherstorferSteinbauer2001}. They established a connection, for $q\in[0, \infty)$, between $L^q$ minimal polynomials on $\E_m$ and specific weighted $L^q$ minimal polynomials on $[0, 1]$, using the canonical change of variables $x=z^m$. In this paper, we will employ a similar change of variables, focusing on the case where $q=\infty$. This approach allows us to establish a relationship between the Chebyshev polynomials of $\E_m$ and weighted minimal polynomials on $[-1, 1]$.
	%Extremal polynomials of the sets $\E_m$ have previously been studied by Peherstorfer and Steinbauer \cite{PeherstorferSteinbauer2001}.  Through the canonical change of variables $x=z^m$, they set up --- for $q\in[ 0, \infty)$ --- a connection between $L^q$ minimal polynomials on $\E_m$ and certain weighted $L^q$ minimal polynomials on $[0, 1]$.  In this article we shall employ a similar change of variables, but to the case of $q=\infty$, and thus relate the Chebyshev polynomials of $\E_m$ to weighted minimal polynomials on $[-1, 1]$.

	%They proved that for $q\in [1,\infty)$, one can relate a large class of $L^q$ extremal polynomials on $\E_{m}$ %(e.g., relative to equilibrium measure) 
%to certain weighted extremal polynomials on $[0, 1]$ an interval. \blue{(say something more here)}
%However, these results do not seem to carry over to the case of $q=\infty$ where only explicit formulas for $T_{nm}^{\E_m}$ appear to be available.
	
	Theorem \ref{thm:widom-factors} handles preimages of a line segment under monomials in the complex plane. The following result gives a complete picture for such preimages of arbitrary quadratic polynomials.
	\begin{theoremm}
	\label{thm:general_quadratic}
		Let $P(z) = z^2+az+b$ for $a, b \in\C$ and form $\E_P = \{z:P(z)\in [-2,2]\}$. Then 
		\begin{equation}
				\cW_{2n,\infty}(\E_P)=2, \quad n\geq 1 
				\label{even} 
		\end{equation}
		and
		\begin{equation}		
				\lim_{n\rightarrow \infty}\cW_{2n+1,\infty}(\E_P)= \sqrt{2\left|c+\sqrt{c^2-4}\right|}, \label{odd}
			\end{equation}
			where $c = b-a^2/4$ and $z+\sqrt{z^2-4}$ maps the exterior of $[-2,2]$ to the exterior of the closed disk of radius $2$ centered at $0$.  
In particular, for $c\in [-2,2]$ we have
		\begin{equation}		
			\lim_{n\rightarrow \infty}\cW_{2n+1,\infty}(\E_P)= 2. 
			\label{odd c}
		\end{equation}
	\end{theoremm}
\begin{remark}	
	Note that $c\in [-2,2]$ if and only if $\E_P$ is connected. Hence, for $\E_P$ a quadratic preimage of $[-2,2]$, we have 
	\[\lim_{n\rightarrow \infty}\cW_{n,\infty}(\E_P)=2\]
	precisely when $\E_P$ is connected. The result in \eqref{odd} also illustrates that we can only hope for a universal upper bound on $\cW_{n, \infty}(\E)$ within the class of compact connected sets.
\end{remark}

	\subsection{Orthogonal polynomials with respect to $\mu_\E$}  \label{OP} %relative to equilibrium measure}
	
	The Chebyshev polynomials are the monic polynomials minimising the $L^\infty$ norm on a given compact set $\E\subset\C$. The same investigation on minimal polynomials can be undertaken for any $L^p$ norm. %Given an infinite compact set $\E\subset\C$ with corresponding equilibrium measure $\mu_\E$, 
	We shall in particular consider the sequence of monic orthogonal polynomials with respect to equilibrium measure $\mu_\E$. These are the polynomials %monic polynomials $P_{n}^{\mu_\E}$, where $\deg(P_n^{\mu_\E}) = n$, of the form
	\[
		P_{n}^{\mu_\E}(z)=z^{n}+\sum_{k=0}^{n-1}c_kz^k
	\]
	which satisfy 
	\[\int P_n^{\mu_\E}(z)\overline{\displaystyle{P_m^{\mu_\E}(z)}}\,d\mu_{\E}(z) = C_n\cdot\delta_{n,m}, \]
	where $C_n>0$ and $\delta_{n,m}$ is the Kronecker delta. %A consequence of this definition is 
	It is well-known and easy to prove that all $P_{n}^{\mu_\E}$ are minimal with respect to the $L^2(\mu_{\E})$-norm in the sense that
	\[
		\Vert P_{n}^{\mu_\E}\Vert_{L^2(\mu_\E)}^2 :=\int |P_{n}^{\mu_\E}(z)|^2 \,d\mu_{\E}(z) = \min_{b_0, b_1, \ldots, b_{n-1}\in\C} \int \left|z^n+\sum_{k=0}^{n-1}b_kz^k\right|^2d\mu_{\E}(z).
	\]
	%Uniqueness and existence of such polynomials is classical. 
	
	In line with \eqref{eq:widom_factor_def}, we define the Widom factors %of degree $n$ on $\E$ 
	corresponding to the $L^2$ minimisers on $\E$ by
	\begin{align}
		\cW_{n,2}(\E) & := \frac{\|\displaystyle{P_n^{\mu_\E}}\|_{L^2(\mu_\E)}}{\Cap(\E)^n}.
	\end{align}
	The name is appropriate since Widom \cite{Widom1969} gave a complete description of the asymptotics of $\cW_{n,2}(\E)$ in the case where $\E$ is a finite union of Jordan curves and arcs of class $C^{2+\epsilon}$. Recent results of Alpan and Zinchenko \cite{AlpanZinchenko2020} suggest a relation between the asymptotics of $\cW_{n,2}(\E)$ and $\cW_{n,\infty}(\E)$. In fact, the method we will use to prove Theorem \ref{thm:widom-factors} relies on results of Bernstein \cite{Bernstein:1930-31} %(to be explained in the appendix) 
	where certain weighted Chebyshev polynomials are related to a class of orthogonal polynomials. This is what initially motivated our study of $\cW_{n,2}(\E_m)$ %for the sets in \eqref{Em}.
	%in the case where
	%\begin{equation}
	%\E_m = \bigl\{z: z^m\in [-2,2]\bigr\}.
	%\label{Em}
	%\end{equation}
	and we have the following result.
	\begin{theoremm}
		With $\E_m$ as defined in \eqref{Em}, % = \{z:z^m\in [-2,2] \}$. 
		we have 
		\[
			\lim_{n\rightarrow \infty}\cW_{n,2}(\E_m)^2 = 2.
		\]
		Moreover, $n\mapsto\cW_{2nm+l,2}(\E_m)^2$
		is monotonically increasing for $0<{l}/{m}<1$ and monotonically decreasing for $1<{l}/{m}<2$.
		\label{thm:widom_factors_orthogonal}
	\end{theoremm}
	Exactly the same monotonicity along subsequences was observed numerically for $\cW_{n,\infty}(\E_m)$. However, without explicit formulas at hand this seems hard to prove. By relating Theorems \ref{thm:widom-factors} and \ref{thm:widom_factors_orthogonal}, we get the following result which complements the results of \cite{AlpanZinchenko2020}.
	\begin{corollary}
		The Widom factors for $\E_m$ satisfy that
		\begin{equation}
			\lim_{n\rightarrow \infty}\cW_{n,2}(\E_m)^2 = \lim_{n\rightarrow \infty}\cW_{n,\infty}(\E_m).
			\label{eq:widom_factor_equality}
		\end{equation}
	\end{corollary}
	Based on the results of \cite{AlpanZinchenko2020}, it was conjectured in \cite{CSZ-review} that if $\E$ is a smooth Jordan arc, then
	\begin{equation}
	\lim_{n\rightarrow \infty}\cW_{n,2}(\E)^2 = \lim_{n\rightarrow \infty}\cW_{n,\infty}(\E).
	\label{W=W}
	\end{equation}
%	In this setting, we have
%	\[
 %       \mu_\E = \frac{1}{2}\left(\frac{\partial}{\partial n_+}G_\E+\frac{\partial}{\partial n_-}G_\E\right).
%        \]
	It is known that \eqref{W=W} holds for straight line segments and circular arcs as in \eqref{eq:thiran_detaille}, see \cite{CSZ-review} and \cite{AlpanZinchenko2020}. In our case we are not dealing with a single Jordan arc but rather a union of Jordan arcs. Nevertheless, our approach may still shed new light on the above conjecture.
	
	\subsection{Outline}
	
This article is organised as follows. 
   In Section \ref{sec:s_property_chebotarev} we discuss the $S$-property and the Chebotarev problem, and use these concepts to present a proof of Theorem \ref{thm:widom_factor_arc}. 
   In Section \ref{sec:quadratic} we consider Chebyshev polynomials on quadratic preimages of $[-2,2]$ and illustrate how Bernstein's method for determining the asymptotic behaviour of weighted Chebyshev polynomials on an interval (summarised in the appendix) can be applied to solve problems in the complex plane. We specifically establish Theorem \ref{thm:widom-factors} for $m=2$ and extend this result to Theorem \ref{thm:general_quadratic}, which completely describes the asymptotic behavior of the Widom factors for quadratic preimages of $[-2,2]$.  
   The findings for $\E_2$ are extended to the setting of all $\E_m$ sets in Section \ref{sec:m_star}, demonstrating the alternation properties of Chebyshev polynomials in this context and providing a complete proof of Theorem \ref{thm:widom-factors}.

%showing that Chebyshev polynomials associated to $\E_m$ for different $m$ coincide. 

%We in particular prove Theorem \ref{thm:widom-factors} for $m=2$ and extend this to Theorem \ref{thm:general_quadratic} which fully characterises the asymptotics of Widom factors on quadratic preimages of $[-2,2]$.

    In Section \ref{sec:orthogonal} we consider polynomials orthogonal with respect to equilibrium measure of $\E_m$ and prove Theorem \ref{thm:widom_factors_orthogonal}, that the associated Widom factors converge to $2$.  We also show that the numerically observed monotonic behavior of the corresponding Chebyshev norms holds for the orthogonal polynomials in this setting. 
    Additional insight into this monotonicity is presented in Section \ref{first}, drawing upon conformal mappings and geometric considerations.
    %Finally, we discuss possible future work in Section \ref{sec:outlook} and with Conjecture \ref{conj} we hypothise a possible link between Shabat polynomial preimages and Widom factor convergence to $2$. This is supported by numerical simulations. We believe that a crucial element for sets to have Widom factors converging to $2$ is the geometric $S$-property.	
    Ultimately, we address potential future research directions in Section \ref{sec:outlook}. With Conjecture \ref{conj}, we propose a potential connection between Shabat polynomial preimages and the convergence of Widom factors to $2$. This conjecture finds support in numerical simulations. We hold the view that a fundamental prerequisite for sets to exhibit such Widom factors lies in the geometric $S$-property.
%A further intuitive understanding of this monotonicity is provided in Section \ref{first}, based on conformal maps and geometric considerations.
			
	\section{The $S$-property and Chebotarev sets}
	\label{sec:s_property_chebotarev}
	We now direct our focus towards compact sets $\E\subset\C$ exhibiting a certain symmetry property which turns out to be of importance regarding the convergence behaviour $\lim_{n\rightarrow \infty}\cW_{n}(\E)=2$. In fact, this is exactly what we will use to prove Theorem \ref{thm:widom_factor_arc}.
	
	The symmetry property in question is called the $S$-property and was introduced by Stahl in the 1980s to study certain extremal domains in the complex plane. The ``$S$'' in the name is therefore ambiguous; it can be read as \emph{symmetry} but also as \emph{Stahl}. We will present a simplified version of the $S$-property, adapted to our needs, following \cite[Definition 2]{Stahl2010}. For a more comprehensive exploration of its connection to extremal domains of meromorphic functions, we direct the reader to \cite{Stahl85-I-II, Stahl85-III, Stahl2012}.		
%	We will state a simplified form of the $S$-property based on \cite[Definition 2]{Stahl2010} that suites our purposes and refer the reader to \cite{Stahl85-I-II,Stahl85-III,Stahl2012} for a more detailed discussion on its relation to extremal domains of meromorphic functions. 
	%As stated here, the definition of the $S$-property is most closely resemblent of the one given in \cite{Martinez-Finkelshtein_2011}.

    \begin{definition}          \label{def:s-property}
        Let $\E\subset \C$ be a compact set with $\Cap(\E)>0$ and suppose $\C\setminus \E$ is connected. Assume further that there exists a subset $\E_0\subset \E$ with $\Cap(\E_0) = 0$ such that
        \[
           \E\setminus \E_0 = \bigcup_{i\in I}\gamma_i
        \]
        where the $\gamma_i$'s are disjoint open analytic Jordan arcs and $I\subset \N$. Then $\E$ is said to satify the $S$-property if
        \begin{equation}
			\frac{\partial G_\E}{\partial n_+}(z) = \frac{\partial G_\E}{\partial n_-}(z)
			\label{eq:equal_green}
		\end{equation}
        holds for all $z\in \E\setminus \E_0$, where $n_+$ and $n_-$ denote the unit normals from each side of the arcs.
    \end{definition}
    \begin{remark}
    The assumption of the arcs $\gamma_i$ being analytic is redundant, as mild smoothness conditions coupled with the fulfillment of \eqref{eq:equal_green} imply the analyticity of these arcs. This observation is mentioned in \cite{Stahl2010} without providing a formal proof.    
     %The assumption that the arcs $\gamma_i$ are analytic is superfluous since mild smoothness assumptions of the arcs $\gamma_i$ together with \eqref{eq:equal_green} being satisfied implies analyticity of the arcs $\gamma_i$. This fact is commented on in \cite{Stahl2010} without proof. 

     The initial question to address in the context of reducing the smoothness of individual arcs pertains to the existence of the normal derivatives in \eqref{eq:equal_green}. The proof of \cite[Proposition 2.2]{Totik2014_2} invokes the Kellogg--Warschawski theorem \cite[Theorem 3.6]{Pommerenke} to establish that if the coordinates of the arcs are differentiable and their derivatives are H\"{o}lder continuous, then the normal derivatives of $G_\E$ exist along each $\gamma_i$. Consequently, it is reasonable to discuss the fulfillment of \eqref{eq:equal_green} for arcs belonging to the class $C^{2+\epsilon}$.
     %The first question needed to be settled if the smoothness of the individual arcs is to be reduced is whether the normal derivatives in \eqref{eq:equal_green} actually exist. It is shown in the proof of \cite[Proposition 2.2]{Totik2014_2} using the Kellogg-Warschawski theorem \cite[Theorem 3.6]{Pommerenke} that if the arcs are differentiable and their derivatives are H\"{o}lder continuous, then the normal derivatives of $G_\E$ exists on $\gamma_i$. It makes sense in particular to talk of \eqref{eq:equal_green} being satsfied for $C^{2+\epsilon}$ arcs. 
     In \cite[Theorem 1.4]{Alpan2022}, it is established that with this smoothness assumption, the arcs $\gamma_i$ indeed become analytic when \eqref{eq:equal_green} is satisfied.
     %In \cite[Theorem 1.4]{Alpan2022}, it is shown that under this assumption the arcs $\gamma_i$ are actually analytic when \eqref{eq:equal_green} holds.
    \end{remark}
     The concept of \emph{Chebotarev sets} is closely related to the $S$-property. These sets are characterized by having minimal capacity while being constrained to include a specific collection of points. To elaborate, the Chebotarev problem, originally posed as a question to P\'olya \cite{Polya}, revolves around the quest for the following set:	
	%Connected to the $S$-property are so-called \emph{Chebotarev sets} which are sets of minimal capacity conditioned to encompass a certain collection of points. More specifically, the Chebotarev problem -- which was posed as a question to P\'olya \cite{Polya} -- concerns finding the following set. 
	\begin{definition}
	Given a finite number of points $\alpha_1,\ldots,\alpha_m\in \C$, the compact connected set $\E$ that contains these points and has minimal logarithmic capacity among all such sets is called the {Chebotarev set} of $\alpha_1,\ldots,\alpha_m$. It is denoted by $\E(\alpha_1,\dotsc,\alpha_m)$.
	\end{definition}

That such sets exist and are unique was proven by Gr\"{o}tzsch \cite{Grotzsch:1930}. He also characterised the Chebotarev sets in terms of the behaviour of certain quadratic differentials. %From the uniqueness property it follows that we can call this the Chebotarev set containing $\alpha_1,\dotsc,\alpha_m$. 	
	Stahl \cite[Theorem 11]{Stahl2012} proved that any solution to a Chebotarev problem must satisfy the $S$-property. %of Definition \ref{def:s-property}. 
	In fact, he gave both necessary and sufficient conditions for a set to be a Chebotarev set in terms of the $S$-property and additional geometric conditions. 
 	%with the end points of the arcs being among the given points in $\E_d$.
	%, see Theorem 11 of section 7.3 in \cite{Stahl2012}. 
	For further properties of Chebotarev sets, see, e.g., \cite[Chapter 1]{Kuzmina1982}. The case of $m=3$ is studied in detail in this monograph.  
	
	Interestingly, there is also a relation between Chebotarev sets and polynomial preimages of intervals. Schiefermayr \cite{Schiefermayr:2014} proved that if $P$ is a polynomial and $P^{-1}([-1,1])$ is connected, then this preimage is the solution to a Chebotarev problem. More specifically, if $\alpha_1,\ldots,\alpha_m$ is an enumeration of the distinct simple zeros of $P^2-1$ then
	\[P^{-1}([-1,1]) = \E(\alpha_1,\dotsc,\alpha_m).\]
In particular, any such polynomial preimage satisfies the $S$-property.	 
	
%	In \cite{Stahl2012}, H. Stahl found a connection between a set possessing the $S$-property and the Chebotarev problem. In particular he proved that a set is the solution to the Chebotarev problem if and only if it satisfies the $S$-property and is of the form specified in Equation \eqref{eq:s_property_sets} with the end points of the arc being among the given points, see Theorem 11 of section 7.3 in \cite{Stahl2012}. 
	
%	In relation to the $S$-property, Alpan \cite[Theorem 1.2]{Alpan2022} showed that if $\E$ is a Jordan arc of class $C^{2+\epsilon}$ then
%	\begin{equation}
%		\limsup_{n\rightarrow \infty}\cW_{n,\infty}(\E)\leq 2
%		\label{eq:alpan_limsup}
%	\end{equation}
%	with equality if and only if \eqref{eq:equal_green} is satisfied at every interior point of $\E$. In fact, in the case that \eqref{eq:equal_green} is satisfied at interior points of $\E$ it follows from the remark after Definition \ref{def:s-property} that $\E$ must be analytic and hence has the $S$-property. 
  
 By combining Stahl's results on minimal capacity with recent results of Alpan, % result on the limiting behaviour of the Widom factors, 
	we can easily prove that equality is possible in \eqref{Alpan} only for straight line segments. In this case, as we know, $\cW_{n,\infty}(\E)=2$ for every $n$.

	\begin{proof}[Proof of Theorem \ref{thm:widom_factor_arc}]
		Let $\E$ be a Jordan arc of class $C^{2+\epsilon}$, connecting two complex points $a$ and $b$. The fact that 
		\[\limsup_{n\rightarrow \infty}\cW_{n,\infty}(\E)\leq 2\]
		is precisely the content of \cite[Theorem 1.2]{Alpan2022}. For equality to hold, it must be so that \eqref{eq:equal_green} holds at all interior points of the arc. As explained in the remark following Definition \ref{def:s-property}, this implies that $\E$ is an analytic Jordan arc possessing the $S$-property. By \cite[Theorem 11]{Stahl2012}, $\E$ is therefore the solution to the Chebotarev problem corresponding to the points $a$ and $b$. Hence $\E(a,b) = [a,b]$, where $[a,b]$ denotes the straight line segment in $\C$ between $a$ and $b$ (see, e.g., \cite{Kuzmina1982}). This completes the proof.
	\end{proof}
	In the remainder of this article, we will consider polynomial preimages of $[-2,2]$. All the sets to be considered have in common that they satisfy the $S$-property and hence are minimal sets for a Chebotarev problem.
	
	\section{Quadratic preimages}
	\label{sec:quadratic}
The general framework is sets of the form
		\[
		 \E_m := \bigl\{z: z^m\in [-2,2]\bigr\}.
		\]
These are star-shaped connected sets which are invariant under rotations by $\pi/m$ radians, see Figures \ref{fig-E2}--\ref{fig-E11}. Furthermore, \cite[Theorem 2]{Schiefermayr:2014} implies that $\E_m$ is the Chebotarev set corresponding to the collection of points
		\[
			\bigl\{2^{1/m}e^{i \pi k/m}: k=1,\dotsc,2m \bigr\}.
		\]
		
		To begin with, we consider the case where $m=2$ even though our results are true for any $m\in\N$. The motivation for this is that we more transparently can provide an intuition for the problem when $m=2$. %in particular discuss why the classical 
		%and discuss which methods are fruitful when studying these kinds of problems. 
		We end the section by also considering the setting of a general quadratic preimage. 
		
		Consider the set
		\[
			\E_2 = \bigl\{z:z^2\in [-2,2]\bigr\} = \bigl[-\sqrt{2},\sqrt{2}\,\bigr]\cup i\bigl[-\sqrt{2},\sqrt{2}\,\bigr],
		\]
		which has the shape of a ``plus sign''.
		Equation \eqref{eq:green_capacity_polynomial_preimage} implies that the Green's function for $\overline{\C}\setminus\E_2$ with pole at $\infty$ is given by
		\[G_{\E_2}(z) = \frac{G_{[-2,2]}(z^2)}{2} = \log|z|+o(1),\quad z\rightarrow \infty.\]
		From this we also see that $\Cap(\E_2) = 1$ and hence 
		\[
		\cW_{n,\infty}(\E_2) = \|T_n^{\E_2}\|_{\E_2}.
		\] 
		
		Recall now that a lemma from \cite{KamoBorodin1994} states that if $\E$ is an infinite compact set and $P$ a polynomial of degree $m$ with leading coefficient $a_m$, then
		\begin{equation}
			T_{nm}^{P^{-1}(\E)} = (T_n^{\E}\circ P)/{a_m^{n}} .
			\label{eq:kamo_borodin}
		\end{equation} 
		For a recent proof, see \cite{CSZ-III}. In our setting this immediately implies that
		\[
			T_{2n}^{\E_2}(z) = T_{n}^{[-2,2]}(z^2),
		\]
		as was also proven in \cite{PeherstorferSteinbauer2001}. 
		Therefore, all the Chebyshev polynomials of even degree for $\E_2$ can be determined explicitly. This further implies that we can calculate ``half'' of the norms, that is, 
		\[
			\cW_{2n}(\E_2) = \|T_{2n}^{\E_2}\|_{\E_2} = \|T_{n}^{[-2,2]}\|_{[-2,2]} = 2.
		\]
		
		What is left to determine are the Chebyshev polynomials of odd degree for $\E_2$. The set $\E_2$ is symmetric %posses symmetry in the way that 
		since it is invariant under rotations by $\pi/2$ radians. More precisely,
		\[
			i\E_2 = \{iz: z\in \E_2\} = \E_2,
		\]
		and by uniqueness of the Chebyshev polynomials it thus follows that
		\[
			(-i)^nT_{n}^{\E_2}(iz) = T_{n}^{\E_2}(z).
		\]
		From this relation we get certain conditions on the coefficients of $T_{2n+1}^{\E_2}$. Several of them will vanish and considering degrees of the form $4n+l$ with $l\in\{1, 3\}$, we have
		\begin{align}
		\begin{split}
			T_{4n+l}^{\E_2}(z) &= z^{4n+l}+\sum_{k=0}^{n-1}a_k z^{4k+l}, 
			\quad a_k\in \bbR. 
			%\\
			%T_{4n+3}^{\E_2}(z) & = z^{4n+3}+\sum_{k=0}^{n-1}b_kz^{4k+3},
			%\quad b_k\in \bbR.
		\end{split}
		\label{eq:chebyshev_odd_sequences}
		\end{align}
		%where $a_k\in\bbR$ and $b_k\in\bbR$. 
		In particular, $T_{4n+l}^{\E_2}$ has a zero at the origin of order at least $l$.
		%$1$ or $3$, respectively. 
		However, as we shall see below, this is the precise order of the zero. 
		
Before stating our main characterisation of $T_{4n+l}^{\E_2}$, we recall that an alternating set for a function $f:I\rightarrow \bbC$ is an ordered sequence of points $\{x_k\}\subset I$ such that
		\[f(x_k)= \sigma(-1)^{k}\|f\|_I,\quad |\sigma|=1.\]
		\begin{lemma}
			For $l\in \{1,3\}$, the Chebyshev polynomial $T_{4n+l}^{\E_2}$ is characterised by alternation in the following way:
			\begin{itemize}
				\item $T_{4n+l}^{\E_2}$ has alternating sets consisting of $n+1$ points on each of the sets $i^{k}[ 0,\sqrt{2} ]$, $k=0,1,2, 3.$ \vspace{0.1cm}
				\item $T_{4n+l}^{\E_2}$ can be represented as in \eqref{eq:chebyshev_odd_sequences}.
			\end{itemize}
			\label{lem:alternation_lemma_deg_2}
		\end{lemma}
		\begin{proof}
			We have already motivated why \eqref{eq:chebyshev_odd_sequences} should hold for $T_{4n+l}^{\E_2}$. This also implies that $T_{4n+l}^{\E_2}$ is purely real on the real axis and purely imaginary on the imaginary axis. Now, basic theory for Chebyshev polynomials implies that $T_{4n+l}^{\E_2}$ has at least $4n+l+1$ extremal points (see, e.g., \cite[Lemma 2.5.3]{Lorentz66}). Due to symmetry, there are at least $n+1$ extremal points on each of the rays 
			\[
			i^{k}[0,\sqrt{2}], \quad k=0, 1, 2, 3.
			\]
			It is enough to consider alternation on one of these rays since, by symmetry, the situation is entirely analogous on each of the rays. We therefore fix our attention to $[0,\sqrt{2}]$ where $T_{4n+l}^{\E_2}$ is real valued. If there were two adjacent extremal points on $[0,\sqrt{2}]$ with the same sign, then this would imply that the derivative $(T_{4n+l}^{\E_2})'$ should have at least $4\left(n+1\right)\geq 4n+l+1$ zeros. But this is a contradiction since the degree of the derivative is $4n+l-1$.
			
			To prove the other direction, we apply the intermediate value theorem in the following way. Suppose $Q$ is a polynomial which can be represented as in \eqref{eq:chebyshev_odd_sequences}. Then $Q$ will be real-valued on the real line and imaginary-valued on the imaginary axis. Assume further that $Q$ possesses alternating sets as in the statement of the lemma. If $Q\neq T_{4n+l}^{\E_2}$, then
			\[\|Q\|_{\E_2} > \|T_{4n+l}^{\E_2}\|_{\E_2}
			\]
			and it follows from the intermediate value theorem that $Q-T_{4n+l}^{\E_2}$ must have a zero between any two consecutive points in the alternating sets for $Q$. But this amounts to at least $4n+l$
%			\[4n+l\]
			zeros in addition to the zero of multiplicity $l$ at $0$. Hence the number of zeros would be greater than the degree of $Q-T_{4n+l}^{\E_2}$ which is impossible. Therefore, we conclude that $Q = T_{4n+l}^{\E_2}$. 
		\end{proof}
		
		One implication of this lemma is the fact that all the zeros of $T_{2n+1}^{\E_2}$ are simple, except the one at $z=0$ for $T_{4n+3}^{\E_2}$. A further implication is that we can determine the first few Chebyshev polynomials for $\E_2$ explicitly:
		\begin{align*}
			T_{1}^{\E_2}(z)& = z,\\
			T_{3}^{\E_2}(z)&=z^3,\\
			T_{5}^{\E_2}(z) &= z^5-\frac{5\left(3-\frac{15^{2/3}}{\sqrt[3]{9+4\sqrt{6}}}+\sqrt[3]{15(9+4\sqrt{6})}\right)^4}{5184}z.
		\end{align*}
		Since already the expression for $T_5^{\E_2}$ gets rather complicated, there seems to be no simple closed form for $T_{2n+1}^{\E_2}$ when $n$ is large.
		
		Lemma \ref{lem:alternation_lemma_deg_2} also implies that the image of $\E_2$ under $T_{2n+1}^{\E_2}$ is again a plus-shaped set, namely % of the form
		\[
			T_{4n+l}^{\E_2}(\E_2) = \bigcup_{k=0}^{3}i^k\bigl[0,\|T_{4n+l}^{\E_2}\|_{\E_2}\bigr] = \frac{\|T_{4n+l}^{\E_2}\|_{\E_2}}{\sqrt{2}}\E_2.
		\]
%		A key player for us in the discussion to follow will be the preimage
To further understand the behaviour of $\|T_{4n+l}^{\E_2}\|_{\E_2}$, we introduce the preimage
		\begin{equation} \label{En}
			\E_2^{n, l} := (T_{4n+l}^{\E_2})^{-1}\bigl(T_{4n+l}^{\E_2}(\E_2)\bigr) =  (T_{4n+l}^{\E_2})^{-1}\left(\bigcup_{k=0}^{3}i^k\bigl[0,\|T_{4n+l}^{\E_2}\|_{\E_2}\bigr]\right)
		\end{equation}
		which clearly contains $\E_2$ as a subset.
%		We define the full preimage as $\E_{2}^{n,l}$, that is, 
%		\[
%			\E_{2}^{n,l} := (T_{4n+l}^{\E_2})^{-1}\left(\bigcup_{k=0}^{3}i^k[0,\|T_{4n+l}^{\E_2}\|_{\E_2}]\right).
%		\]
		Due to the alternating property described in Lemma \ref{lem:alternation_lemma_deg_2}, we deduce that $\E_{2}^{n,l}$ is connected and a finite union of Jordan arcs. These in turn intersect at the zeros of $T_{4n+l}^{\E_2}$. Moreover, $\E_{2}^{n,l}$ is nothing but the Chebotarev set corresponding to the points
		\[
			(T_{4n+l}^{\E_2})^{-1}\Bigl(\bigl\{i^k\|T_{4n+l}^{\E_2}\|_{\E_2}:k=0,1,2,3\bigr\}\Bigr).
		\]
Applying \eqref{eq:green_capacity_polynomial_preimage} and using the fact that $\Cap(\alpha \E) = |\alpha|\Cap(\E)$ for any $\alpha\in \C$, %\eqref{eq:green_capacity_polynomial_preimage}, 
we see that
		\begin{equation}
		\label{cap nl}
			\Cap(\E_{2}^{n,l}) = \Cap\left(\frac{\|T_{4n+l}^{\E_2}\|_{\E_2}}{\sqrt{2}}\E_2\right)^{1/{4n+l}} = \left(\frac{\|T_{4n+l}^{\E_2}\|_{\E_2}}{\sqrt{2}}\right)^{1/{4n+l}}.
		\end{equation}
		Rearranging this equality leads to the identity
		\begin{equation}
			 \|T_{4n+l}^{\E_2}\|_{\E_2} = \sqrt{2}\Cap(\E_{2}^{n,l})^{4n+l}.
			 \label{eq:capacity_cheb_norm}
		\end{equation}
		Since $\E_{2}^{n,l}\supset\E_2$, we have $\Cap(\E_2^{n,l})\geq \Cap(\E_2) = 1$ and conclude that %the Chebyshev polynomials are lower bounded on $\E_2$ in the way that
		\begin{equation}
			\|T_{4n+l}^{\E_2}\|_{\E_2}\geq \sqrt{2}.
			\label{lb}
		\end{equation}

		Note that the reasoning outlined above follows the same method of proof that we employed to prove the lower bounds of Szeg\H{o} and Schiefermayr in the introduction.
		%, see \eqref{eq:szego_inequality} and \eqref{eq:schiefermayr_inequality}. 
		However, the conclusion is not equally strong. The lower bound in \eqref{lb} is saturated for $T_{1}^{\E_2}$, but otherwise not. Using $zT_{2n}^{\E_2}(z)$ as a trial polynomial, we find that % Producing a similar upper bound is not difficult. In fact, since $\|T_{2n}^{\E_2}\|_{\E_2} = 2$ we find that
		\[\|T_{2n+1}^{\E_2}\|_{\E_2}\leq \|zT_{2n}^{\E_2}\|_{\E_2}\leq \sqrt{2}\|T_{2n}^{\E_2}\|_{\E_2} = 2^{3/2}.\]
		This upper bound is only saturated by $T_{3}^{\E_2}$. To close the gap between $\sqrt{2}$ and $2^{3/2}$ in the limit as $n\to\infty$, additional insight is needed. We shall settle this issue in the next subsection.
		
		At this point, we would like to highlight an intriguing observation that, to date, remains open. %We mention in passing a curious observation that we so far have been unable to prove. 
Recall that the Remez algorithm \cite{Remez34-I, Remez34-II}
		%due to E. Remez from 1934 
		uses the theory of alternation to numerically compute Chebyshev polynomials of real compact sets. This algorithm was generalised to the complex setting by Tang %in an article from 1988 
		\cite{Tang1988} and further refined by Modersitzki and Fischer \cite{FischerModersitzki1994}. Using an implementation of this generalised algorithm, we can compute the norms of the Chebyshev polynomials for degree up to at least $60$. What materialised for $\E_2$ is illustrated in Figure \ref{fig:norm_plot}.
		\begin{figure}[h!]
			\centering
			\includegraphics[width=0.7\textwidth]{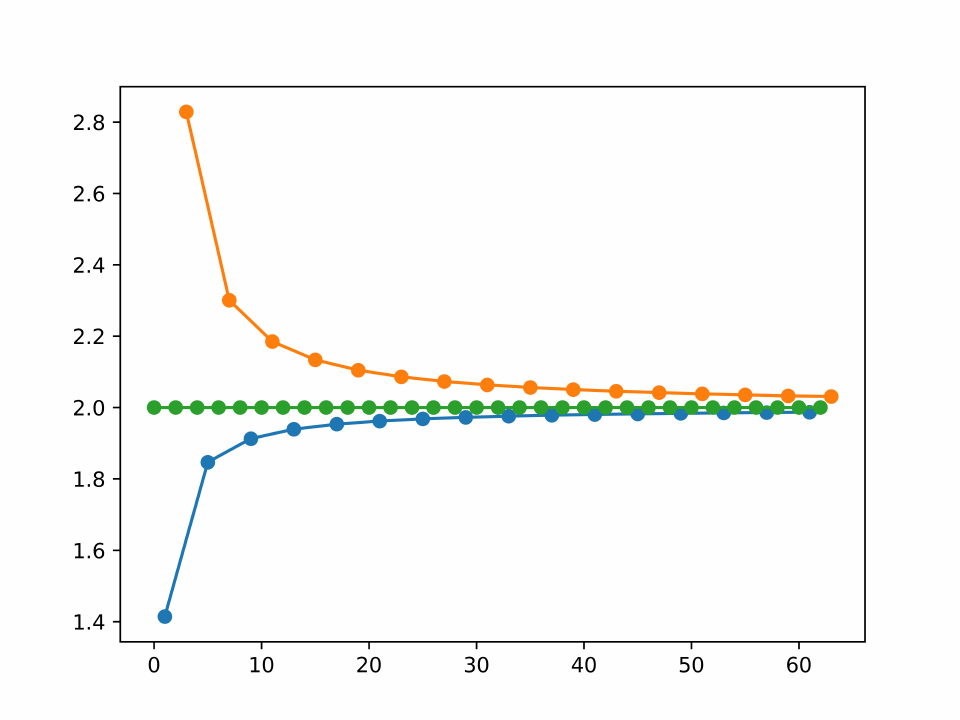}
			\caption{The norm of $T_n^{\E_2}$ plotted as a function of the degree $n$. The three natural subsequences are highlighted in different colors.}
			\label{fig:norm_plot}
		\end{figure}
		As we already know, $\|T_{n}^{\E_2}\|_{\E_2}=2$ for all even $n$. For odd $n$, there are two natural subsequences counting modulo $4$. It appears that $n\mapsto \|T_{4n+1}^{\E_2}\|_{\E_2}$ is monotonically increasing, while $n\mapsto \|T_{4n+3}^{\E_2}\|_{\E_2}$ appears to be monotonically decreasing. Similar patterns emerge for the general $\E_m$ sets. 					
		\subsection{A related weighted problem and the proof of Theorem \ref{thm:widom-factors} for $m=2$}
Rather than employing \eqref{eq:capacity_cheb_norm} and attempting to directly estimate the capacity of $\E_2^{n, l}$, we shall shift the minimax problem to a weighted problem on $\R$ which can, in turn, be solved using a method of Bernstein \cite{Bernstein:1930-31}.
%		Instead of estimating the capacity of $\E_2^{n, l}$, we shall now attack the problem by use of another classical approach, namely {trial polynomials}. An intermediate step in this direction will be to transfer the problem to a weighted problem on $\R$.
		
		%The method that we will apply in the remainder of this paper is based on work of Bernstein \cite{Bernstein:1930-31}. 
		Based on \eqref{eq:chebyshev_odd_sequences}, the Chebyshev polynomials of odd degree for $\E_2$
%		\[\E_2 = \bigl\{z:z^2\in [-2,2]\bigr\}\]
		can be represented as
		\[
			T_{4n+l}^{\E_2}(z) = z^{4n+l}+\sum_{k=0}^{n-1}a_kz^{4k+l} = z^{l}\left(z^{4n}+\sum_{k=0}^{n-1}a_kz^{4k}\right) = z^lQ_{n,l}(z^4),
		\]
		where $Q_{n,l}$ is a real monic polynomial of degree $n$. Note that $Q_{n,l}$ has the property that it minimises
		\[
			\max_{z\in \E_2}| z^l | \bigl\vert Q_{n,l}(z^4) \bigr\vert
		\]
		among all monic polynomials of degree $n$. Since $|z^{1/4}|$ is single-valued, we can change the variables via the transformation $x = z^4/2-1$ and obtain that $Q_{n,l}$ minimises
		\[	
			\max_{x\in [-1,1]}2^{l/4}(x+1)^{{l}/{4}}\left|Q_{n,l}\bigl(2(x+1)\bigr)\right|. 
			%= 2^{n}\sup_{x\in [-1,1]}(x+1)^{{l}/{4}}\left| {R}_{n,l}(x)\right|
		\]
		Put differently, and noting that $2^{-n} Q_{n,l} (2(x+1) )$ is monic, we see that 
		\begin{equation}
			\|T_{4n+l}^{\E_2}\|_{\E_2} = 2^{n+{l}/{4}} 
			\min_{b_0, b_1, \ldots, b_{n-1}\in\R} \max_{x\in [-1,1]}(x+1)^{{l}/{4}}\left| x^n+\sum_{k=0}^{n-1} b_k x^k \right|.
			\label{eq:cheb_on_star_odd_sequence}
		\end{equation}
	%	where the infimum is taken over all monic degree $n$ polynomials.
		
		In the case of $l=0$, we are back at the classical problem solved by Chebyshev and the polynomials $T_n$ of \eqref{Tn}. For $l=2$, the solution is given by the so-called Chebyshev polynomials of the third kind, denoted $V_n$. This is easily seen by noting that 
		\[
		   (x+1)^{1/2} V_n(x)=\sqrt{2} \cos\bigl( (n+1/2)\arccos(x)\bigr)
		\]    
		oscillates on $[-1, 1]$ between precisely $n+1$ extrema of equal magnitude. However, we are interested in the cases $l=1, 3$ where no such simple formulas seem to be available.

		The key is to recall that $T_n$ and $V_n$ are also orthogonal polynomials. Both are special cases of the Jacobi polynomials $P_n^{(\alpha, \beta)}$ which have been studied in detail and are orthogonal with respect to $(1-x)^\alpha (1+x)^\beta$ on $[-1, 1]$. 
		If we have a non-negative weight function $w$ defined on $[-1, 1]$, we will use the notation ${T_n^{w}}$ to represent the minimiser of
		\[
		 \max_{x\in [-1,1]}w(x)|P_n(x)|,
		\]   
where $P_n$ ranges over all monic polynomials of degree $n$. This minimizer is uniquely defined for sufficiently large $n$ under the conditions that $w$ does not vanish at infinitely many points and $wP$ is bounded for some polynomial $P$.
		%Given a weight function $w\geq 0$ on $[-1, 1]$, we shall henceforth denote by ${T_n^{w}}$ the minimiser of among all monic polynomials $P_n$ of degree $n$. This minimiser is uniquely defined for large enough values of $n$ if $w$ is non-vanishing at infinitely many points and $wP$ is bounded above for some polynomial $P$. 
		A naive comparison of the $L^\infty$ and $L^2$ norms suggests that ${T_n^w}$ could be related to the monic orthogonal polynomials with respect to $w(x)^2/\sqrt{1-x^2}$ on $[-1,1]$. In our setting, the weight is given by
                \[
                     w_l(x)=(x+1)^{l/4}
                \]      
and we thus search for a relation between ${T_n^{w_l}}$ and $P_n^{(-1/2, \,\beta_l)}$ for $\beta_l=(l-1)/2$. 
		In special cases (such as $l=0, 2$) we have an exact match and the polynomials coincide. But even for general $l$, there are striking similarities between the polynomials. In \cite{Bernstein:1930-31}, Bernstein proved that $w_l P_n^{(-1/2, \, \beta_l)}$ is asymptotically alternating on $[-1,1]$ as $n\to\infty$ and because of that, %By use of this property, he was able to show that 
		$P_n^{(-1/2, \, \beta_l)}$ are excellent trial polynomials when studying the limit of $\Vert w_l T_n^{w_l}\Vert_{[-1,1]}$.

%The original proof by Bernstein is available in Russian and French and hence we believe that a presentation in English can be beneficial for a wider audience.
		
		We shall formulate Bernstein's full result below. As the proof is only available in French or Russian, we decided to include a detailed review of this method in the appendix. 
While our presentation primarily adheres to Bernstein's original proof, we are able to streamline some arguments by integrating findings from Achieser and Chebyshev alongside Bernstein's work.

%Although our presentation mostly follows the original proof of Bernstein, we are able to simplify certain arguments by combining results of Achieser and Chebyshev with those of Bernstein. 	
			
		\begin{theoremm}[Bernstein \cite{Bernstein:1930-31}]
			Suppose $\alpha_k\in \bbR$ and $b_k\in [-1,1]$ for $k=0, 1, \ldots, m$.  Consider a weight function $w:[-1,1]\rightarrow [0,\infty)$ of the form
			\begin{equation}  \label{weight}
				w(x) = w_0(x)\prod_{k=0}^{m}|x-b_k|^{\alpha_k},
			\end{equation}
			where $w_0$ is Riemann integrable and satisfies $1/M \leq w_0(x) \leq M$ for some constant $M\geq 1$. Then %the weighted Chebyshev polynomial
%			\[
%				T_n^{w}(x) = x^{n}+\sum_{k=0}^{n-1}a_kx^k
%			\]
%			which minimises
%			\[
%				\sup_{x\in [-1,1]}w(x)|T_n^w(x)|
%			\]
%			among all monic polynomials of degree $n$ satsifies
			\begin{equation}  \label{B asymp}
				%\sup_{x\in [-1,1]}w(x)|T_n^w(x)|= 
				\Vert w T_n^w \Vert_{[-1, 1]}=
				2^{1-n}\exp\left\{\frac{1}{\pi}\int_{-1}^{1}
				\frac{\log w(x)}{\sqrt{1-x^2}}dx\right\}
				\bigl(1+o(1)\bigr)
			\end{equation}
			as $n\rightarrow \infty$.
			\label{thm:Bernstein}
		\end{theoremm}
%		Since the original proof only appears in print in the French and Russian languages we will provide an English translation in the appendix. Our proof will further feature some minor simplifications as our main aim is the study of Chebyshev polynomials. 
		
		With Bernstein's theorem at hand, we can now prove Theorem \ref{thm:widom-factors}. We isolate the following lemma since it will be used multiple times in the rest of the article.
		\begin{lemma}
			For any $z\in \C$, we have
			\begin{equation}
			\label{int}
				\frac{1}{\pi}\int_{-1}^{1}\frac{\log|x-z|}{\sqrt{1-x^2}}dx = \log\frac{\left|z+\sqrt{z^2-1}\right|}{2},
			\end{equation}
			where $z+\sqrt{z^2-1}$ maps $\C\setminus[-1,1]$ conformally onto $\C\setminus\overline{\D}$. In particular, for $z\in [-1,1]$ the integral is constantly equal to $-\log 2$.
			\label{lem:log_integral_equilibrium_measure}
		\end{lemma}
		
		\begin{proof}
			Recall that the Green's function for $\overline{\C}\setminus[-1,1]$ with pole at $\infty$ is given by
			\[
				G_{[-1,1]}(z) = \log\left|z+\sqrt{z^2-1}\right|.
			\]
			Since the integral on the left-hand side of \eqref{int} is nothing but the potential for the equilibrium measure of $[-1, 1]$, we conclude that 
			\[
				\frac{1}{\pi}\int_{-1}^{1}\frac{\log|x-z|}{\sqrt{1-x^2}}dx =G_{[-1,1]}(z)+\log\Cap\bigl([-1,1]\bigr)=\log\frac{\left|z+\sqrt{z^2-1}\right|}{2}
			\] 
for $z\notin [-1,1]$. %The proof is thus complete for such $z$. 
If instead $z \in [-1,1]$, we find by continuity of $G_{[-1,1]}$ on $\C$ and monotone convergence that
            \begin{align*}
                \frac{1}{\pi}\int_{-1}^{1}\frac{\log|x-z|}{\sqrt{1-x^2}}dx & = \lim_{\epsilon\downarrow 0}\frac{1}{\pi}\int_{-1}^{1}\frac{\log|x-(z+i\epsilon)|}{\sqrt{1-x^2}}dx\\& = \lim_{\epsilon\downarrow 0}\log\frac{\left|(z+i\epsilon)+\sqrt{(z+i\epsilon)^2-1}\right|}{2}=\log\frac{\left|z+\sqrt{z^2-1}\right|}{2} = -\log 2.
            \end{align*}
            This completes the proof.
		\end{proof}
		
We are now in position to prove the limiting result for $\E_2$. As we shall see in Section \ref{sec:m_star}, the method of proof actually handles all cases in Theorem \ref{thm:widom-factors}. 

		\begin{proof}[Proof of Theorem \ref{thm:widom-factors} for $m=2$]
			Let $w_l(x) = (x+1)^{{l}/{4}}$ with $l=1,3$. Theorem \ref{thm:Bernstein} implies that
			\[
				\|T_n^{w_l} w_l\|_{[-1,1]} = 2^{1-n}\exp\left\{\frac{l}{4\pi}\int_{-1}^{1}\frac{\log (x+1)}{\sqrt{1-x^2}}dx\right\}\bigl(1+o(1)\bigr)
			\]
			as $n\rightarrow \infty$. Applying Lemma \ref{lem:log_integral_equilibrium_measure} gives us that the the exponential on the right-hand side reduces to
			\[
				%\exp\left\{\frac{1}{\pi}\int_{-1}^{1}\frac{\log w(x)}{\sqrt{1-x^2}}dx\right\}
				\exp\left\{\frac{l}{4}\cdot\log\frac{1}{2}\right\} = 2^{-{l}/{4}}.
			\]
			By \eqref{eq:cheb_on_star_odd_sequence}, we therefore find that
			%\[
			%	\|T_n^ww\|_{[-1,1]} = 2^{1-n-{l}/{4}}\bigl(1+o(1)\bigr)
			%\]
			%and 
			\[
				\|T_{4n+l}^{\E_2}\|_{\E_2} = 2^{n+{l}/{4}}\|T_n^{w_l} w_l\|_{[-1,1]} = 
				2\bigl(1+o(1)\bigr),
			\]
		        which settles the claim for $\E_2$.
		\end{proof}
		
%Returning to the discussion of Section \ref{first}, we can now state the following direct consequence.
%\begin{corollary}
%Let $\E_2^{n,l}$ be the set as defined in \eqref{En} for $l=1,3$.
%Not only will $\Cap(\E_2^{n,l})\to 1$ but we even have 
%\[
% \Cap(\E_2^{n,l})^{4n+1}\to\sqrt{2} \; \mbox{ as } \; n\to \infty.
%\] 
%\end{corollary}				
		
		\subsection{General quadratic preimages}
		We proceed by considering general quadratic preimages of $[-2, 2]$ and show that our method can be applied to this case as well. Of course, such preimages need not be connected and we should not expect the Widom factors to always converge to $2$. %\blue{add widom reference?}, 
		%However, the asymptotic behavior of $\cW_{n,\infty}$ can still be determined. 
		%In determining the asymptotic behaviour of the Widom factors for a general quadratic preimage of $[-2,2]$ it is enough to consider
		As in Theorem \ref{thm:general_quadratic}, we shall consider sets of the form
		\begin{equation}
		\E_P=\bigl\{z: P(z)\in [-2,2]\bigr\}, 
		\label{eq:quadratic-preimage}
		\end{equation}
		where $P(z)=z^2+az+b$ and $a,b\in \bbC$.  Now, let us present a proof of this result.		%\begin{theoremm}
		%	Let $\E$ be given as in \eqref{eq:quadratic-preimage}. Then
		%	\begin{align}
		%		\cW_{2n,\infty}(\E)&=2,\\
		%		\lim_{n\rightarrow \infty}\cW_{2n+1,\infty}(\E)&= \sqrt{2\left|c+\sqrt{c^2-4}\right|},
		%	\end{align}
		%	where $c = b-\left(\frac{a}{2}\right)^2$.
		%\end{theoremm}
		\begin{proof}[Proof of Theorem \ref{thm:general_quadratic}]
			By \eqref{eq:green_capacity_polynomial_preimage}, we have $\Cap(\E_P) = 1$ and thus
			\[
			 \cW_{n,\infty}(\E_P) = \|T_{n}^{\E_P}\|_{\E_P}.
			\] 
			The first part of the theorem (i.e., \eqref{even} for even degree) is easily handled by \eqref{eq:kamo_borodin}. In order to prove \eqref{odd}, we note that
			\[z^2+az+b = \left(z+\frac{a}{2}\right)^2+c, \quad c=b-a^2/4.\]
			Therefore, it suffices to consider the case of
			\[\E_{P_c} = \bigl\{z : z^2+c\in [-2,2]\bigr\}.\]
			As $z\in \E_{P_c}$ %\Leftrightarrow 
			if and only if $-z\in\E_{P_c}$, we deduce that
			\[
				T_{2n+1}^{\E_{P_c}}(z)=z^{2n+1}+\sum_{k=0}^{n-1}a_kz^{2k+1}, 
				\quad a_k\in\C. 
			\]
			Hence, since $|\sqrt{z}|=\vert z\vert^{1/2}$ is single-valued, we can apply the change of variables $x = (z^2+c)/2$ leading to
			\[
				\|T_{2n+1}^{\E_{P_c}}\|_{\E_{P_c}} = 
				2^{n+{1}/{2}}\min_{b_0, b_1, \ldots, b_{n-1}\in\bbR}
				\max_{x\in [-1,1]}\sqrt{\vert x-c/2 \vert} \left|x^n
				+\sum_{k=0}^{n-1}b_kx^k\right|.
			\]
			Since the weight function $\sqrt{\vert x- c/2\vert}$ is of the form \eqref{weight} --- in fact, we only need $w_0$ for $c\notin [-2, 2]$ --- Theorem \ref{thm:Bernstein} implies that 
			\begin{equation}
				\|T_{2n+1}^{\E_{P_c}}\|_{\E_{P_c}}= 2^{3/2}\exp\left\{\frac{1}{2\pi}\int_{-1}^{1}\frac{\log |x-c/2|}{\sqrt{1-x^2}}dx\right\}\bigl(1+o(1)\bigr)
				\label{eq:asymptotics_quadratic_preimage}
			\end{equation}
			as $n\rightarrow \infty$. Using Lemma \ref{lem:log_integral_equilibrium_measure} to rewrite the exponential term, 
			%	\[
			%	%\frac{1}{\pi}\int_{-1}^{1}\frac{\log|x-c/2|}{\sqrt{1-x^2}}dx = 
			%\pi\log\frac{\bigl\vert{c}/{2}+\sqrt{\left({c}/{2}\right)^2-1}\bigr\vert}{2}.
			%\]
			%where again $w+\sqrt{w^2-1}$ maps $\C\setminus[-1,1]$ onto $\{z:|z|>1\}$. 
we conclude that % Equation \eqref{eq:asymptotics_quadratic_preimage} implies that
			\[
				\lim_{n\to\infty}\|T_{2n+1}^{\E_{P_c}}\|_{\E_{P_c}} = 
				2^{3/2}\cdot \sqrt{\frac{\bigl\vert {c}/{2}+\sqrt{\left({c}/{2}\right)^2-1}\bigr\vert}{2}} = \sqrt{2\left|c+\sqrt{c^2-4}\right|}
			\]
			and the proof is complete.
		\end{proof}

\begin{remark}
As the attentive reader may have noticed, we employ a different change of variables in the proof of Theorem \ref{thm:general_quadratic} compared to what is used for $\E_2$.  This is because we no longer have the same degree of symmetry at our disposal.  In principle, one could have applied the change of variables $x=z^2/2$ for $\E_2$ as well.  However, it seems more appropriate to leverage all the available structure, especially considering the distinct behavior of $\Vert T_{4n+l}^{\E_2}\Vert_{\E_2}$ for $l=1, 3$.
\end{remark}		
		
As pointed out in the introduction, $\E_P$ is only connected when $c\in[-2, 2]$. For otherwise the critical value of $z^2+c$ lies outside $[-2, 2]$. The corresponding sets have the shape of a ``stretched plus'' and we appear to loose the monotonicity of
\[
   n \mapsto \Vert T_{4n+l}^{\E_P} \Vert_{\E_P}, \quad l=1, 3
\]		
whenever $c\neq 0$, see Figures \ref{fig_quadratic_preimage2}--\ref{fig:quadratic_max_norms2}.	

\begin{figure}[h!]
			\centering
			\begin{minipage}{.4\textwidth}
		  		\centering
  				\includegraphics[width=1.2\linewidth]{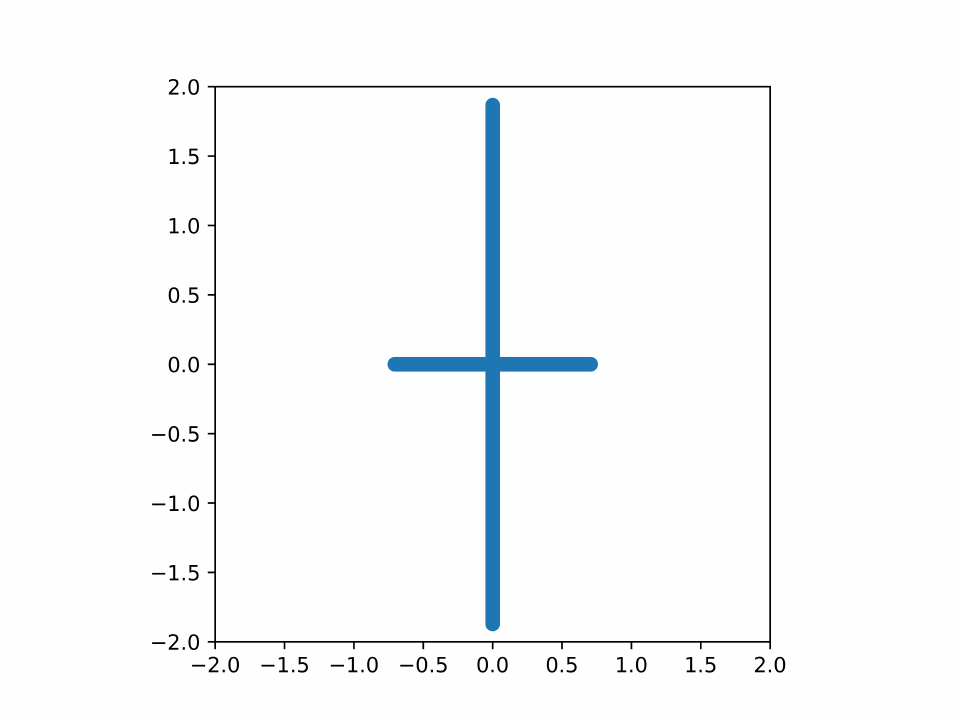}
				\caption{$\E_{P_c}$ for $c=-3/2$}
				\label{fig_quadratic_preimage2}
			\end{minipage}%
			\begin{minipage}{.6\textwidth}
				\centering
				\includegraphics[width=.8\linewidth]{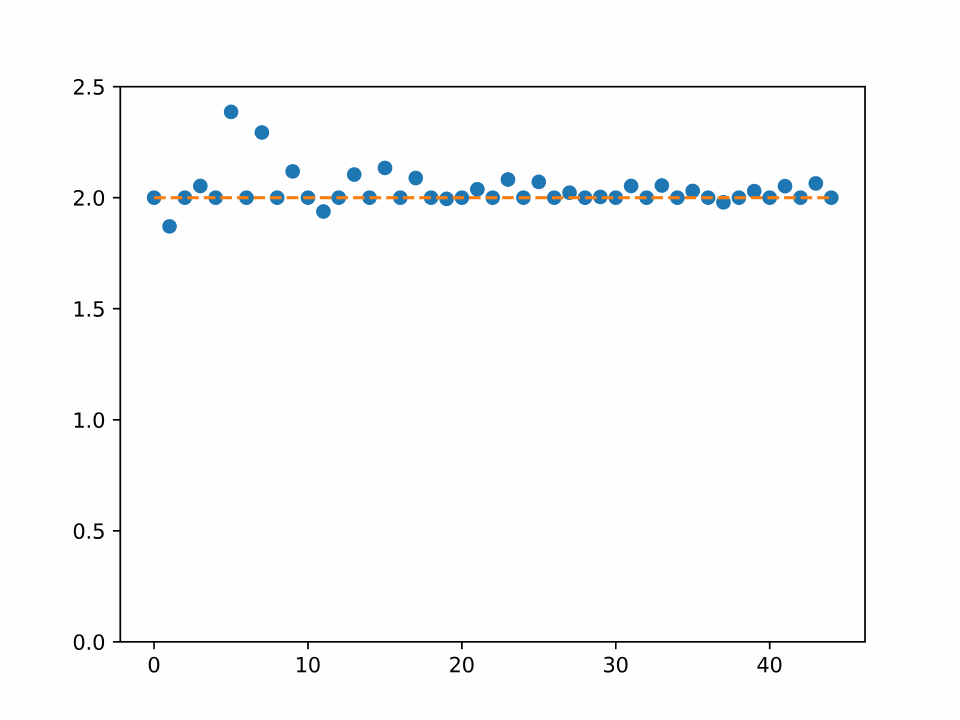}
				\caption{The norms of $T_{n}^{\E_{P}}$ for $c=-3/2$}
				\label{fig:quadratic_max_norms2}
			\end{minipage}
		\end{figure}	

For $c\in\R\setminus[-2, 2]$, the set $\E_P$ is either purely real or purely imaginary and consists of two disjoint intervals of the same length. This case is well-understood and the fact that $n\mapsto \cW_{n, \infty}(\E_P)$ is asymptotically periodic goes back to Achieser \cite{Ach}. 
Remarkably, this pattern %of asymptotic periodicity 
persists for all non-real values of $c$ where $\E_P$ comprises two disjoint analytic Jordan arcs of opposite sign.
 See Figures \ref{fig_quadratic_preimage3}--\ref{fig:quadratic_max_norms3} for such an example including the corresponding norms.
				
	\begin{figure}[h!]
			\centering
			\begin{minipage}{.4\textwidth}
		  		\centering
  				\includegraphics[width=1.2\linewidth]{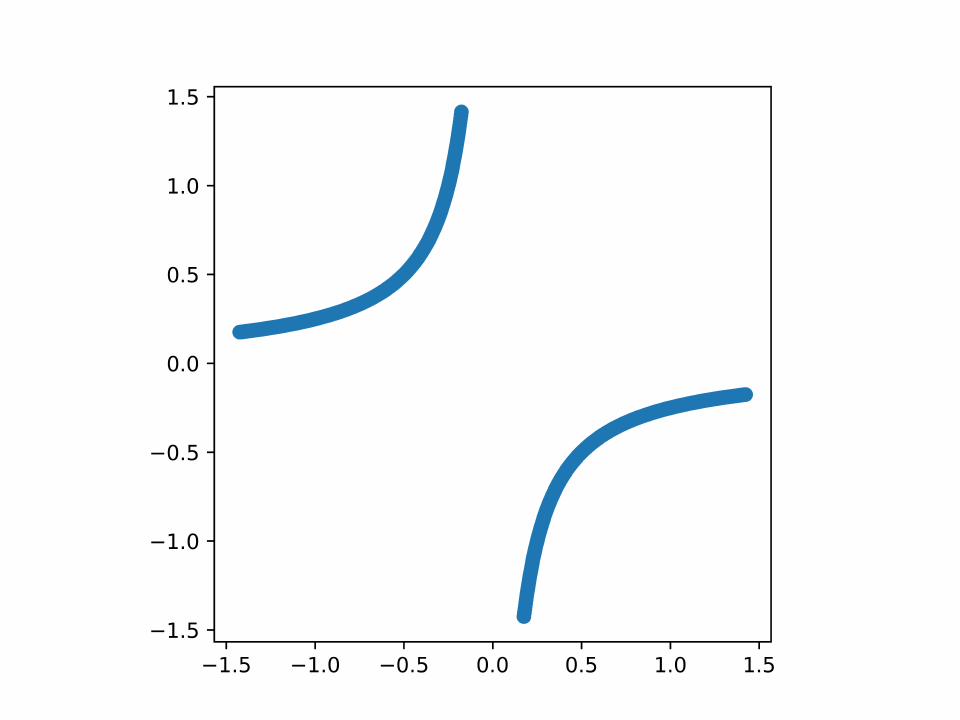}
				\caption{$\E_{P_c}$ for $c=i/2$}
				\label{fig_quadratic_preimage3}
			\end{minipage}%
			\begin{minipage}{.6\textwidth}
				\centering
				\includegraphics[width=.8\linewidth]{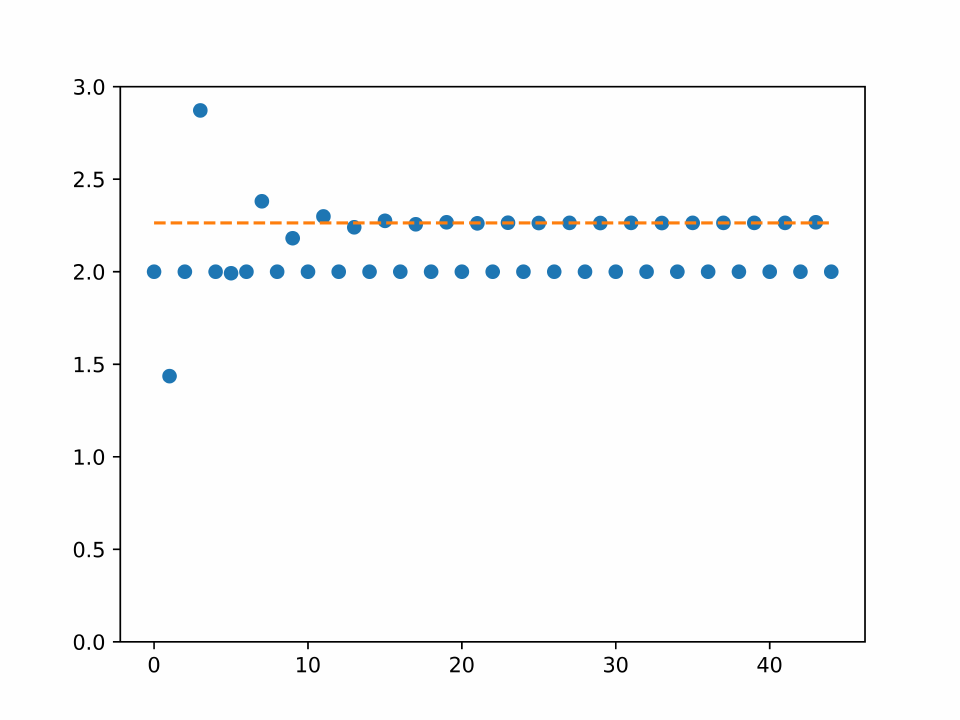}
				\caption{The norms of $T_{n}^{\E_{P}}$ for $c=i/2$}
				\label{fig:quadratic_max_norms3}
			\end{minipage}
		\end{figure}			
	
	\section{General $m$-stars}
	\label{sec:m_star}
	The aim of this section is to generalise the content of Section \ref{sec:quadratic} to the general setting of $m\geq 2$ %where
%	\[
%		\E_m = \bigl\{z:z^m\in [-2,2]\bigr\}
%	\]
	and in particular prove Theorem \ref{thm:widom-factors}. As the strategy is very similar, 
	%to the case of $m=2$, 
	we shall not provide all details in the proof. Let us first state the alternation result that still holds in this setting.
	
	\begin{lemma}
	\label{sym}
		For $n,l\in \bbN$ and $0\leq l<2m$, the polynomial $T_{2nm+l}^{\E_m}$  is characterised by the following two properties:
		\begin{itemize}
			\item for $x\in [0,2^{1/m}]$ and $k=0, 1, \ldots, 2m-1$, the polynomial $T_{2nm+l}^{\E_m}(e^{i\pi k/m} x)$ has an alternating set consisting of $n+1$ points, \vspace{0.1cm}	
			\item $T_{2nm+l}^{\E_m}$ is of the form
			\begin{equation} \label{eq:cheb_representation_star}
				T_{2nm+l}^{\E_m}(z) = z^{2nm+l}+\sum_{k=0}^{n-1}a_{2km+l}z^{2km+l},
				\quad a_{2km+l}\in\R.
			\end{equation}
		\end{itemize}
	\end{lemma}
	\begin{proof}
		We first show that $T_{2nm+l}^{\E_m}$ can be represented as in \eqref{eq:cheb_representation_star}. Since $z\in \E_m$ if and only if $e^{i\pi/m}z\in \E_m$, we find that
		\[
			T_{2nm+l}^{\E_m}(e^{i\pi/m}z) = e^{2\pi i n}e^{i\pi l/m}T_{2nm+l}^{\E_m}(z).
		\]
		Hence, by writing
		\[
			T_{2nm+l}^{\E_m}(z) = z^{2nm+l}+\sum_{k=0}^{2nm+l-1}a_kz^{k},
		\]
		we see that
		\[
			e^{i\pi (k-l)/m}a_k=a_k.
		\]
		For $a_k$ to be non-zero, it is thus necessary that $e^{i\pi (k-l)/m}=1$. In other words, $k=2mj+l$ for some integer $j$.  We also note that $\E_m$ is conjugate symmetric in the sense that $z\in \E_m$ if and only if $\overline{z}\in \E_m$.  This implies that $a_k\in \bbR$ for all $k$.
		
		To prove the alternating property, one argues in the same way as in Lemma \ref{lem:alternation_lemma_deg_2} and we leave the details to the reader.
	\end{proof}	

    There is a certain overlap between the Chebyshev polynomials corresponding to different $\E_m$ sets. Indeed, for $m, k\in \N$ we have
    \[
       \E_{mk} = \bigl\{z:z^{mk}\in [-2,2]\bigr\} = \bigl\{z: z^k\in \E_m\bigr\}
    \]
and we thus gather from \eqref{eq:kamo_borodin} that 
    \begin{equation}
    \label{kn=nk}
       T_{kn}^{\E_{mk}}(z) = T_{n}^{\E_m}(z^k), \quad n\geq 1.
    \end{equation}
This can also be proven directly using Lemma \ref{sym} which we now briefly illustrate. Due to \eqref{eq:cheb_representation_star}, we have
    \[T_{k(2nm+l)}^{\E_{mk}}(z) = (z^k)^lQ(z^k),\]
where $Q$ is a monic polynomial of degree $2nm$. Since $z^k$ maps $\E_{mk}$ onto $\E_m$, we conclude that $Q$ is the minimal monic polynomial for
    \[\min_{b_0,b_1,\dotsc,b_{2nm-1}\in\R}\max_{z\in \E_m}\left|z^l\left(z^{2nm}+\sum_{k=0}^{2nm-1}b_kz^{k}\right)\right|.\]
    But this expression is minimised by $T_{2nm+l}^{\E_m}$ and hence
    \begin{equation}
    \label{kmn}
    T_{k(2nm+l)}^{\E_{mk}}(z) = T_{2nm+l}^{\E_m}(z^k), \quad n\geq 1.
    \end{equation}
As a consequence of \eqref{kmn}, we see that $\Vert T_{k(2nm+l)}^{\E_{mk}}\Vert_{\E_{mk}}=\Vert T_{2nm+l}^{\E_m}\Vert_{\E_m}$ for all $n\geq 1$.  Monotonicity of one of these subsequences therefore implies monotonicity of the other.
    
	We are now in position to prove Theorem \ref{thm:widom-factors} in full generality.
	
	\begin{proof}[Proof of Theorem \ref{thm:widom-factors}]
		We apply the change of variables $x = z^{2m}/2-1$, note that $|z^{1/2m}|$ is single-valued, and obtain 
		\[
			\|T_{2nm+l}^{\E_m}\|_{\E_m} = 2^{n+{l}/{2m}}
			\min_{b_0, b_1, \ldots, b_{n-1}\in\R} \max_{x\in [-1,1]}
			|x+1|^{{l}/{2m}}\left|x^n+\sum_{k=0}^{n-1}b_kx^k\right|.
		\]
		Theorem \ref{thm:Bernstein} gives us that the minimum on the right-hand side behaves like
		\[
			%\inf_{b_k}\max_{x\in [-1,1]}|x+1|^{\frac{l}{2m}}\left|x^n+\sum_{k=0}^{n-1}b_kx^k\right|= 
			2^{1-n}\exp\left\{\frac{l}{2\pi m}\int_{-1}^{1}\frac{\log|x+1|}{\sqrt{1-x^2}} dx\right\}\bigl(1+o(1)\bigr)
		\]
		as $n\rightarrow \infty$, and Lemma \ref{lem:log_integral_equilibrium_measure} implies that		\[
			\exp\left\{\frac{l}{2\pi m}\int_{-1}^{1}\frac{\log|x+1|}{\sqrt{1-x^2}}dx\right\} = 2^{-{l}/{2m}}.
		\]
		Combining these formulas yields \eqref{limit n}.
		
		To prove \eqref{limit m}, we simply note that the representation in \eqref{eq:cheb_representation_star} implies that
		\[
		 T_{2m-1}^{\E_m}(z)=z^{2m-1}, \quad m\geq 1
		\]
		and hence $\|T_{2m-1}^{\E_m}\|_{\E_m} = 2^{(2m-1)/m} = 2^{2-1/m}$.		 
		%\[
		%\|T_{2nm+l}^{\E_m}\|_{\E_m}\rightarrow 2
		%\]
		%as $n\rightarrow \infty$.
	\end{proof}

\begin{remark}
	Throughout this exposition, we opted to consider polynomial preimages of $[-2, 2]$, the canonical interval of capacity one. Consequently, $\E_m$ consistently possesses an even number of edges and symmetry in both the real and imaginary axes. Had we chosen the interval $[0, 4]$ as our starting point, our sets would have encompassed all symmetric star graphs, including those with an odd number of edges. It is worth noting that the method of proof extends seamlessly to this scenario, and the equivalent of Theorem \ref{thm:widom-factors} holds for such sets as well. To be specific, if we let
	\[
	   	\Sm_m=\bigl\{ z : z^m\in [0, 4] \bigr\},
	\]
then $\cW_{mn, \infty}(\Sm_m)=2$ for all $n\geq 1$ and $\lim_{n\to\infty}\cW_n(\Sm_m)=2$.
	
	%Throughout the presentation, we decided to consider polynomial preimages of $[-2, 2]$ which is the canonical interval of capacity one.  As a consequence, $\E_m$ always has an even number of edges and is symmetric both in the real and imaginary axes.  If we instead had taken the interval $[0, 4]$ as a starting point, our sets would have included all symmetric star graphs, also the ones with an odd number of edges.  Needless to say, the method of proof carries over to this case as well and the counterpart of Theorem \ref{thm:widom-factors} for such sets is also true.
\end{remark}
	
As a consequence of \eqref{limit m}, we see that it is possible to construct Widom factors arbitrarily close to $4$. 
Furthermore, since $T_l^{\E_m}(z)=z^l$ for all $0<l<2m$, we have
\[
   \Vert T_l^{\E_m} \Vert_{\E_m}=2^{l/m}, \quad l<2m
\]
and this value is $<2$ for $0<l/m<1$ and $>2$ for $l/m>1$.   
As in the the case of $m=2$, this result is consistent with the numerics presented in Figure \ref{fig:norm_plots_m_3}. Again, the plot suggests monotonicity along the natural subsequences (mod $2m$). 
	
	\begin{figure}[!h]
		\centering
		\includegraphics[width=0.7\textwidth]{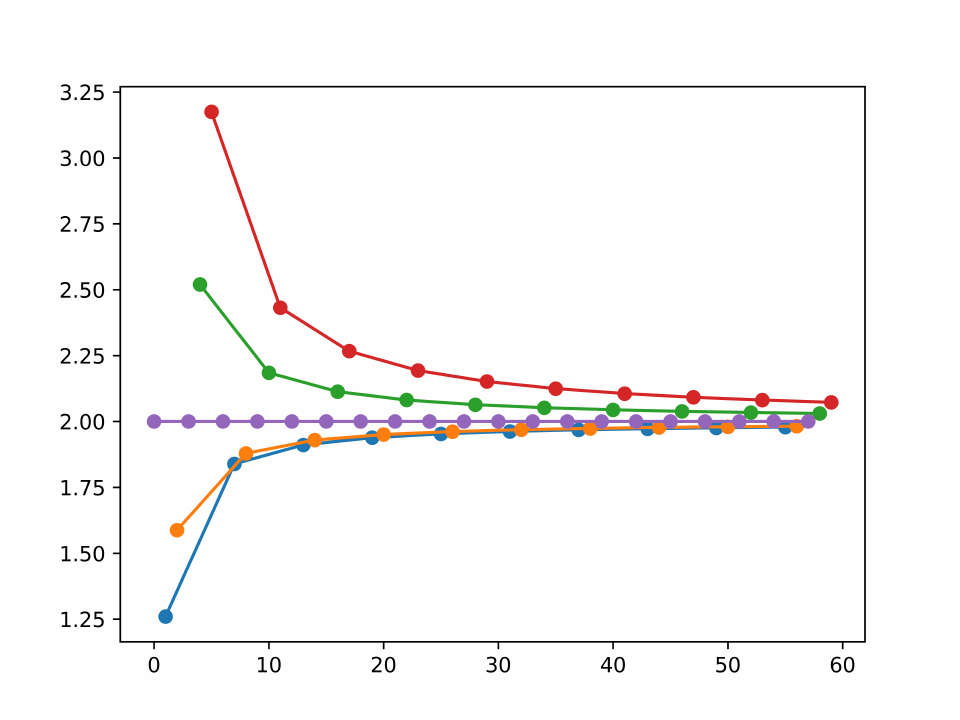}
		\caption{The norms of $\|T_{6n+l}^{\E_3}\|_{\E_3}$ for $0\leq n\leq 4$ and $0\leq l \leq 5$.}
		\label{fig:norm_plots_m_3}
	\end{figure}

	In the next section we shall prove that if $P_{n}^{\mu_{\E_m}}$ denotes the monic orthogonal polynomial with respect to equilibrium measure $\mu_{\E_m}$, then
	\begin{equation}  \label{L2}
	   n\mapsto \|P_{2nm+l}^{\mu_{\E_m}}\|^2_{L^2(\mu_{\E_m})}
	\end{equation}
	is monotonic for any fixed $l$. In fact, this sequence increases when $l/m <1$ and decreases for $l/m>1$. This matches the numerically suggested behavior for the sup-norms of the Chebyshev polynomials. % on $\E_m$.
	
%	\begin{figure}
%		\centering
%		\includegraphics[width=0.9\textwidth]{norm_plot_m_4.png}
%%		\caption{The norm of $\|T_{8n+l}^{\E_4}\|_{\E_4}$ for $0\leq n\leq 4$ and $0\leq l \leq 7$.}
%	\end{figure}
	%\section{$L^2$ problem}

	\section{Related orthogonal polynomials}
	\label{sec:orthogonal}
	We now turn our attention to the orthogonal polynomials with respect to the equilibrium measure of $\E_m$. Recall that if
%	\[\E_m = \{z:z^m\in [-2,2]\}.\]
%	If $\E$ is a compact subset of the complex plane and 
$\mu$ is a probability measure supported on the outer boundary of $\E_m$, then $\mu=\mu_{\E_m}$ %equals the equilibrium measure relative to this set if and only if
precisely when \eqref{rel} holds.
%	\[
%		\int_{\bbC}\log|z-\zeta|d\mu(\zeta) = G_{\E}(z)+\log \Cap(\E).
%	\]
	With this result at hand, we can explicitly determine $\mu_{\E_m}$
	%the equilibrium measure of $\E_m$ 
	(see also \cite{PeherstorferSteinbauer2001}). %\blue{add reference to peherstorfer-steinbauer}
	
	\begin{lemma}  \label{equi}
		Let $\E_m$ be defined as in \eqref{Em}. The equilibrium measure of $\E_m$ is absolutely continuous with respect to arc length measure and given by
		\begin{equation}
			d\mu_{\E_m}(z) = \frac{|z^{m-1}|}{\pi \sqrt{4-z^{2m}}} |dz|,
			\quad z\in\E_m.
			\label{eq:equilibrium_measure}
		\end{equation}
	\end{lemma}
	\begin{proof}
		A straightforward computation shows that the measure defined in \eqref{eq:equilibrium_measure} is a probability measure on $\E_m$. 
		%We omit the proof that $\mu_{\E_m}(\bbC) = 1$. 
		We proceed to show that it satisfies \eqref{rel}.
		Note that if $\zeta, z\in \bbC$ then
		\[\log| \zeta-z^m | = \sum \log| \zeta^{1/m}-z |, \]
		where the sum is taken over all the $m$th roots of $\zeta$. Using this, we can compute the potential for the measure in question and arrive at 
		%defined in Equation \eqref{eq:equilibrium_measure}.
		\begin{align*}
		\frac{1}{\pi}\int_{\E_m}\log| x-z | \frac{\vert x^{m-1}\vert}{\sqrt{4-x^{2m}}} \vert dx\vert 
		&=\frac{1}{\pi}\sum_{k=1}^{2m}\int_{0}^{2^{1/m}}
		\log|e^{i\pi k/m}x-z |\frac{x^{m-1} }{\sqrt{4-x^{2m}}}dx \\
		&= \frac{1}{m\pi}\sum_{k=1}^{m}\int_{-2}^{2}
		\frac{\log| e^{i\pi k/m}t^{1/m}-z |}{\sqrt{4-t^2}}dt \\
		& = \frac{1}{m\pi}\int_{-2}^{2}\frac{\log|t- z^m|}{\sqrt{4-t^2}}dt
		= \frac{1}{m}G_{[-2,2]}(z^m).
		\end{align*}
		The last equality follows from the fact that $dt/(\pi\sqrt{4-t^2})$ is the equilibrium measure of $[-2,2]$ which is a set of logarithmic capacity one.
		 %and hence the corresponding potential equals Green's function for $[-2,2]$ evaluated at $z^m$. 
		By Lemma \ref{pre}, we have
		\[
			%\frac{1}{m\pi}\int_{-2}^{2}\frac{\log|t- z^m|}{\sqrt{4-t^2}}dt=
			\frac{1}{m}G_{[-2,2]}(z^m)=G_{\E_m}(z)
		\]
		and consequently \eqref{eq:equilibrium_measure} holds true since also $\Cap( \E_m )=1$.
		%hence we see that we are able to reproduce Green's function for $\E_m$ as the potential of $\mu_{\E_m}$ implying that $\mu_{\E_m}$ is in fact the equilibrium measure in question.
	\end{proof}
As in Section \ref{OP}, we shall use the notation $P_{n}^{\mu_{\E_m}}$ for the monic orthogonal polynomials with respect to $\mu_{\E_m}$. Recall that these polynomials minimise the integrals
	\[
		\int_{\E_m}\left|P_{n}\right|^2d\mu_{\E_m}
	\]
	among all monic polynomials of degree $n$. Just as in Lemma \ref{sym}, one can use the symmetries of $\E_m$ to show that for $n, l \in\N$ and $0\leq l<2m$, we have
	%Again certain subsequences arise due to the symmetries of the both set and the equilibrium measure $\mu_{\E_m}$.
	%\begin{lemma}
		%Given $n,l\in \bbN$ the polynomial $P_{2nm+l}^{\mu_{\E_m}}$ has the form
		\begin{equation}
		\label{Pn mu}
		   P_{2nm+l}^{\mu_{\E_m}}(z) = z^{2nm+l}+\sum_{k=0}^{n-1}a_{2km+l}z^{2km+l},
		   \quad a_{2km+l}\in \bbR.
		\end{equation}  
	%\end{lemma}
We are now ready to prove Theorem \ref{thm:widom_factors_orthogonal}.
%	\begin{theoremm}
%		With $\E_m$ as above it holds that
%		\[\|P_{2nm+l}^{\E_m}\|_{L^2(\E_m)}^2\rightarrow 2\]
%		as $n\rightarrow \infty$ for any $l$, furthermore the convergence is monotone in $n$. It is increasing if $0<\frac{l}{m}<1$ and decreasing if $1<\frac{l}{m}<2$.
%	\end{theoremm}
	
	\begin{proof}[Proof of Theorem \ref{thm:widom_factors_orthogonal}]
%		By using the fact that $|z^{1/(2m)}|$ is single valued we find a minimisation relation between $P_{2nm+l}^{\E_m}$ and the Jacobi polynomial $P_n^{(-\frac{1}{2},\frac{l}{m}-\frac{1}{2})}$ relative to the weight $(1-t)^{-\frac{1}{2}}(1+t)^{\frac{l}{m}-\frac{1}{2}}$ on $(-1,1)$. This is done via the following computation
Fix $m\geq 2$, let $n\in\N$ and suppose $0\leq l<2m$. By Lemma \ref{equi} and with $P_{2nm+l}^{\mu_{\E_m}}$ as in \eqref{Pn mu}, we have
	\begin{align}  \label{int P}
	\notag
	\|P_{2nm+l}^{\mu_{\E_m}}\|_{L^2(\mu_{\E_m})}^2 & = 		
	\frac{1}{\pi}\int_{\E_m}\left|z^{2nm+l}+\sum_{k=0}^{n-1}a_{2km+l}z^{2km+l}\right|^2 
	\frac{\vert z^{m-1}\vert}{\sqrt{4-z^{2m}}}\vert dz \vert \\
	\notag
	&=\frac{2m}{\pi}\int_{[0, {2^{1/m}}]} 
	\left|z^{2nm}+\sum_{k=0}^{n-1}a_{2km+l}z^{2km}\right|^2  
	\frac{\vert z^{m+2l-1}\vert}{\sqrt{4-z^{2m}}} \vert dz\vert \\
	&=\frac{1}{\pi}\int_{0}^{4}\left(x^n+\sum_{k=0}^{n-1}a_{2km+l}x^k\right)^2
	\frac{x^{l/m-1/2}}{\sqrt{4-x}} dx.
%		&=\frac{2^{2n+\frac{l}{m}+1}}{2\pi}\int_{-1}^{1}|P_n^{(-\frac{1}{2},\frac{l}{m}-\frac{1}{2})}(t)|^2(1-t)^{-\frac{1}{2}}(1+t)^{\frac{l}{m}-\frac{1}{2}}dt.
	\end{align}
	Realising that the above integral is minimised precisely for the Jacobi polynomials (suitably rescaled to $[0, 4]$), we can compute it explicitly. Following Szeg\H{o} \cite[Chapter IV]{Sz}, we adopt the notation
	\[
	   P_n^{(\alpha, \beta)}(x)= {n+\alpha \choose n}
	   {}_2F_1\left( -n, n+\alpha+\beta+1; \alpha+1; \frac{1-x}{2} \right)
	\]   
and note that the leading coefficient of $P_n^{(\alpha, \beta)}$ is given by %\cite[(4.21.6)]{Sz}
        \[
           2^{-n}{2n+\alpha+\beta \choose n}=
           \frac{\Gamma(2n+\alpha+\beta+1)}{2^n \Gamma(n+1)\Gamma(n+\alpha+\beta+1)}.
        \]        	
%With $\beta_l=l/m-1/2$, 
The integral in \eqref{int P} can thus be written as
	%be found in the literature \blue{reference to ismail?} and is given by
	\begin{multline}
	\label{Gs}
	\quad \frac{2^{4n+l/m} \Gamma(n+1)^2 \Gamma(n+l/m)^2}{\pi \Gamma(2n+l/m)^2}
	\int_{-1}^{1}\left( P_n^{(-1/2, l/m-1/2)}(x)\right)^2(1-x)^{-{1}/{2}}(1+x)^{l/m-1/2}dt \\
	= \frac{2^{4n+2l/m} \Gamma(n+{1}/{2}) \Gamma(n+1) \Gamma(n+{l}/{m})
	\Gamma(n+{l}/{m}+1/2)} {\pi (2n+{l}/{m}) \Gamma(2n+l/m)^2} \quad
	\end{multline}
since %\cite[(4.3.3)]{Sz}
\[
   \int_{-1}^1 \left( P_n^{(\alpha, \beta)}(x) \right)^2 (1-x)^\alpha (1+x)^\beta dx=
   \frac{2^{\alpha+\beta+1} \Gamma(n+\alpha+1) \Gamma(n+\beta+1)}
   {(2n+\alpha+\beta+1) \Gamma(n+1) \Gamma(n+\alpha+\beta+1)}.
\]      	
Applying Legendre's duplication formula, the expression in \eqref{Gs} reduces to
%	\[
%		\Gamma(z)\Gamma(z+1/2) = 2^{1-2z}\sqrt{\pi}\Gamma(2z).
%	\]
\[
   \frac{2 \Gamma(2n+1) \Gamma(2n+2l/m)}{(2n+l/m) \Gamma(2n+l/m)^2}
\]
and recalling that $\Gamma(x+a)/\Gamma(x+b)\sim x^{a-b}$ as $n\to\infty$, we deduce that
\[
   \|P_{2nm+l}^{\mu_{\E_m}}\|_{L^2(\mu_{\E_m})}^2 \to 2 \; \mbox{ as } \; n\to\infty.
\]   
   
%	Combining these two result with the previous formula gives us that
%	\begin{align*}
%		\|P_{2nm+l}^{\E_m}\|_{L^2(\mu_{\E_m})}^2 & = \frac{2^{2n+\frac{l}{m}+1}}{2\pi}\cdot \frac{(n!)^22^{2n}}{[(\frac{l}{m}+n)_n]^2}\cdot \frac{2^{\frac{l}{m}+1}\Gamma(n+\frac{1}{2})\Gamma(n+\frac{1}{2}+\frac{l}{m})}{n!\Gamma(n+\frac{l}{m})(2n+\frac{l}{m})}\\
%		& = 2\cdot \frac{2^{4n+2\frac{l}{m}}}{2\pi}\frac{n! \Gamma(n+\frac{1}{2})\Gamma(n+\frac{1}{2}+\frac{l}{m})\Gamma(n+\frac{l}{m})}{\Gamma(2n+\frac{l}{m})^2(2n+\frac{l}{m})} \\
%		& = 2\cdot \frac{2^{2n}}{\sqrt{\pi}}\frac{n!\Gamma(n+\frac{1}{2})\Gamma(2n+2\frac{l}{m})}{\Gamma(2n+\frac{l}{m})^2(2n+\frac{l}{m})}=2\cdot \frac{(2n)\Gamma(2n)\Gamma(2n+\frac{2l}{m})}{\Gamma(2n+\frac{l}{m})^2(2n+\frac{l}{m})}\rightarrow 2
%	\end{align*}
%	as $n\rightarrow \infty$. This settles the asymptotics of the norm. 

In order to prove the claimed monotonicity along subsequences, we introduce the quantity
	\[
		\gamma_n(s) := \frac{2 \Gamma(2n+1) \Gamma(2n+2s)}
		{(2n+s) \Gamma(2n+s)^2}, \quad s\geq 0.
	\]
Note that $\|P_{2nm+l}^{\mu_{\E_m}}\|_{L^2(\mu_{\E_m})}^2 = \gamma_{n}({l}/{m})$ and	\begin{align*}
		\frac{\gamma_{n+1}(s)}{\gamma_{n}(s)} &= 
	\frac{\Gamma(2n+3) \Gamma(2n+2s+2) (2n+s) \Gamma(2n+s)^2}  
	{\Gamma(2n+1) \Gamma(2n+2s) (2n+s+2) \Gamma(2n+s+2)^2}\\
		&=\frac{(2n+1) (2n+2) (2n+2s) (2n+2s+1)}
		{(2n+s) (2n+s+1)^2 (2n+s+2)}.
%		&=\frac{(2n+2)^2(2n+1)(2n+2s+1)}{(2n+s+1)^2(2n+s)(2n+s+2)}\\
%		&=\frac{(2n+2)(2n+1)}{(2n+s+1)(2n+s)}\cdot\frac{(2n+2)(2n+2s+1)}{(2n+s+1)(2n+s+2)}\\
%		&=\frac{(2n+2)(2n+1)}{(2n+s+1)(2n+s)}\cdot\frac{(2n)^2+3\cdot(2n)+2+(2\cdot(2n)+4)s}{(2n)^2+3\cdot(2n)+2+s^2+(2\cdot(2n)+3)s}.
	\end{align*}
Since
\[
   \frac{(x+1)(x+2s)}{(x+s)(x+s+1)}=\frac{x^2+2xs+x+2s}{x^2+2xs+x+s^2+s}	
\]   
and
\[
   \frac{2s}{s^2+s}=\frac{2}{s+1}\gtrless 1 \; \mbox{ for } \; s\lessgtr 1,
\]
it follows that $\gamma_{n+1}(s)/\gamma_n(s)\gtrless 1$ for $s\lessgtr 1$. In this way, we have obtained the desired monotonicity statement.   
%   We immediately see that
%	\begin{align*}
%		\frac{\gamma_{n+1}(s)}{\gamma_{n}(s)}&>1	,\quad s\in (0,1)\\
%		\frac{\gamma_{n+1}(s)}{\gamma_{n}(s)}&<1,\quad s\in (1,\infty).
%	\end{align*}
%	In particular $\|P_{2nm+l}^{\E_n}\|_{L^2(\mu_{\E_m})}^2$ increases monotonically if $0<\frac{l}{m}<1$ and decreases monotonically if $1<\frac{l}{m}<2.$
	\end{proof}
	
	The sets $\E_m$ appear to be the first examples beyond the framework of single Jordan arcs for which \eqref{W=W} is fulfilled. We wonder how far one can push the conditions on $\E$ in the conjecture of \cite{CSZ-review}.

		\section{A further investigation regarding $\cW_{2n+1}(\E_2)$}  \label{first}
	The plot in Figure \ref{fig:norm_plot} indicates that $\Vert T_{2n+1}^{\E_2}\Vert_{\E_2}\lessgtr 2$ depending on whether $n$ is even or odd.  As proven in Section \ref{sec:quadratic}, $T_{4n+l}^{\E_2}$ has a zero of order $l$ at $z=0$ (for $l=1, 3$).  However, there is no apparent explanation for why a zero with greater multiplicity would necessarily result in a higher sup-norm.  In this section we shall study the preimages defined in \eqref{En} in more detail and try to shed more light on the pattern revealed by Figure \ref{fig:norm_plot}.
		
%		Motivated by the methods developed in \cite{CSZ-I, CSZ-II}, we initially tried to control the norms of $T_{4n+l}^{\E_2}$ by estimating the capacity of $\E_{2}^{n,l}$. The relation in \eqref{cap nl}--\eqref{eq:capacity_cheb_norm} indicates that this may be a suitable path.
		%can then be used to determine the norm of the Chebyshev polynomial. 
		%This method has proven useful in the determination of the norm asymptotics of Chebyshev polynomials on the real line, see for instance \cite{CSZ-I}. 
		
		As explained in Section \ref{sec:quadratic}, $\E_{2}^{n,l}$ consists of the base set $\E_2$ together with certain ``decorations''.  At all the points where $T_{4n+l}^{\E_2}$ has a zero, a new ``crossing'' will appear and each of the four preimages 
		\begin{equation}
		(T_{4n+l}^{\E_2})^{-1}\Bigl( i^k \bigl[0, \|T_{4n+l}^{\E_2}\|_{\E_2}\bigr]\Bigr), \quad k=0, 1, 2, 3
		\label{preim}
		\end{equation}
consists of $4n+l$ arcs. While $2n+1$ of these arcs lie on $\E_2$, the remaining $2n+l-1$ are orthogonal to $\E_2$ (except for the $l-1$ arcs emanating at the origin).  It can be shown that every arc in $\E_2^{n, l}$ has harmonic measure equal to ${1}/{4(4n+l)}$, for instance, by examining a conformal map from $\C\setminus \E_2^{n,l}$ onto $\C\setminus \overline{\D}$. The typical shape of an $\E_2^{n, l}$ set is illustrated in Figures \ref{fig-E21}--\ref{fig-E23} below. 
		\begin{figure}[h!]
			\centering
			\begin{minipage}{.5\textwidth}
		  		\centering
  				\includegraphics[width=.9\linewidth]{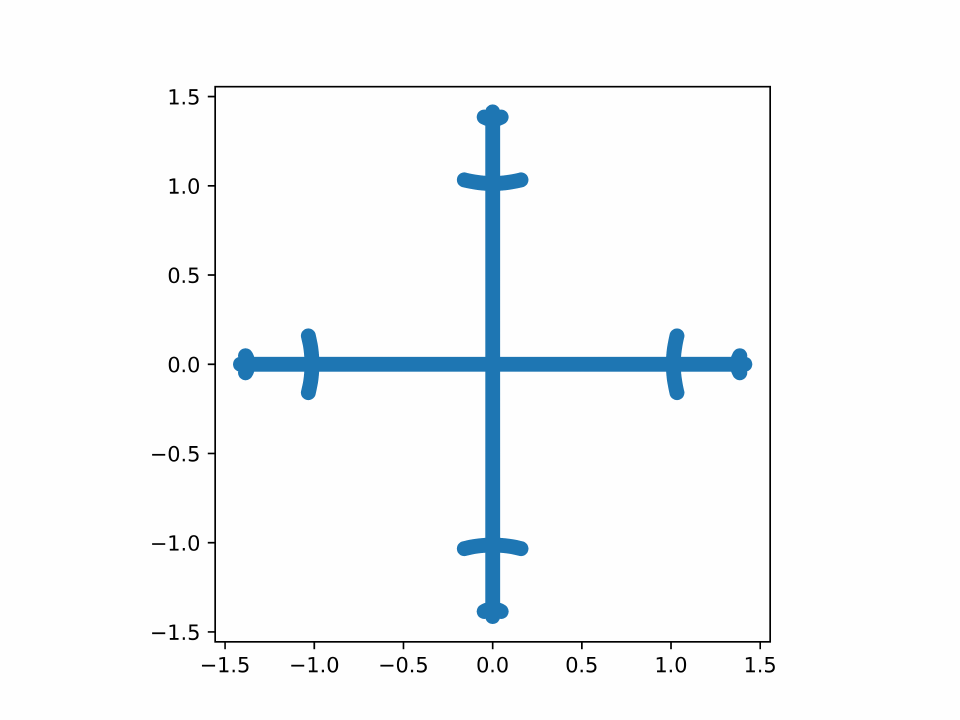}
				\caption{$\E_2^{2,1}$}
				\label{fig-E21}
			\end{minipage}%
			\begin{minipage}{.5\textwidth}
				\centering
				\includegraphics[width=.9\linewidth]{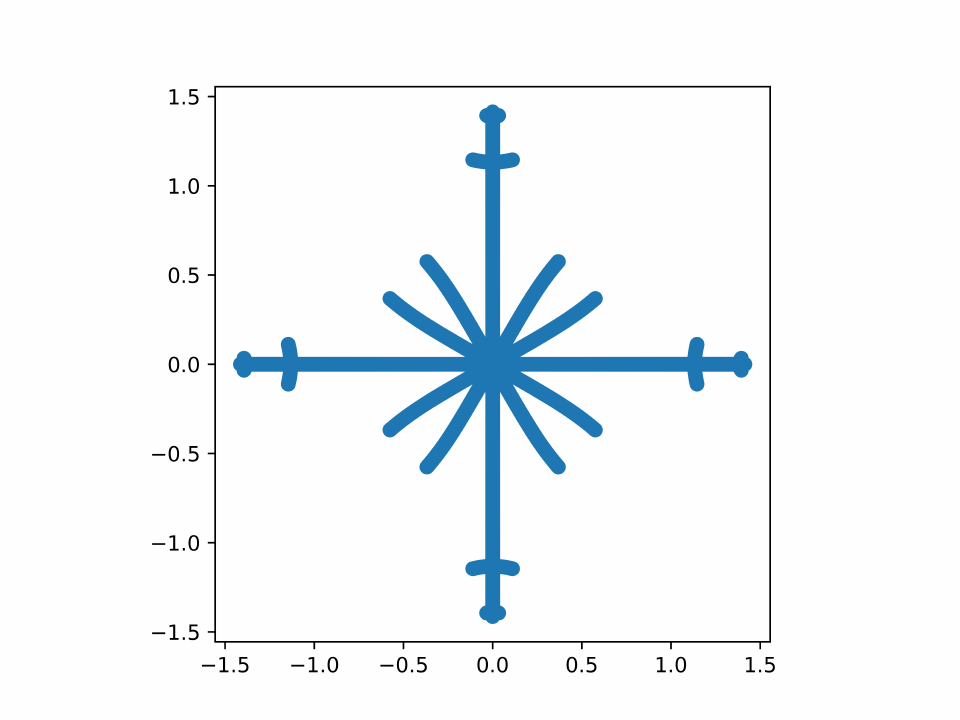}
				\caption{$\E_{2}^{2,3}$}
				\label{fig-E23}
			\end{minipage}
		\end{figure}	

With \eqref{eq:capacity_cheb_norm} in mind, we see that understanding the behaviour of $\Vert T_{4n+l}^{\E_2} \Vert_{\E_2}$ equates to effectively estimating the capacity of $\E_2^{n, l}$.  Since the zeros of $T_{4n+l}^{\E_2}$ are hard to determine explicitly, %(we only know that they distribute according to equilibrium measure of $\E_2$ as $n\to\infty$)
it seems difficult to estimate this capacity directly.  We get a better picture by considering the conformal map
		\[
			\Phi(z) = \sqrt{{z^2}/{2}+\sqrt{\left({z^2}/{2}\right)^2-1}\,}	
		\]
which maps $\C\setminus\E_2$ onto $\C\setminus\overline{\D}$.  When applying $\Phi$, the complement of $\E_2^{n, l}$ is mapped onto the complement of the closed unit disk accompanied by certain ``protrusions'', see Figures \ref{fig-E21-image}--\ref{fig-E23-image}. 		
		\begin{figure}[h!]
		\begin{minipage}{.5\textwidth}
		  		\centering
  				\includegraphics[width=.9\linewidth]{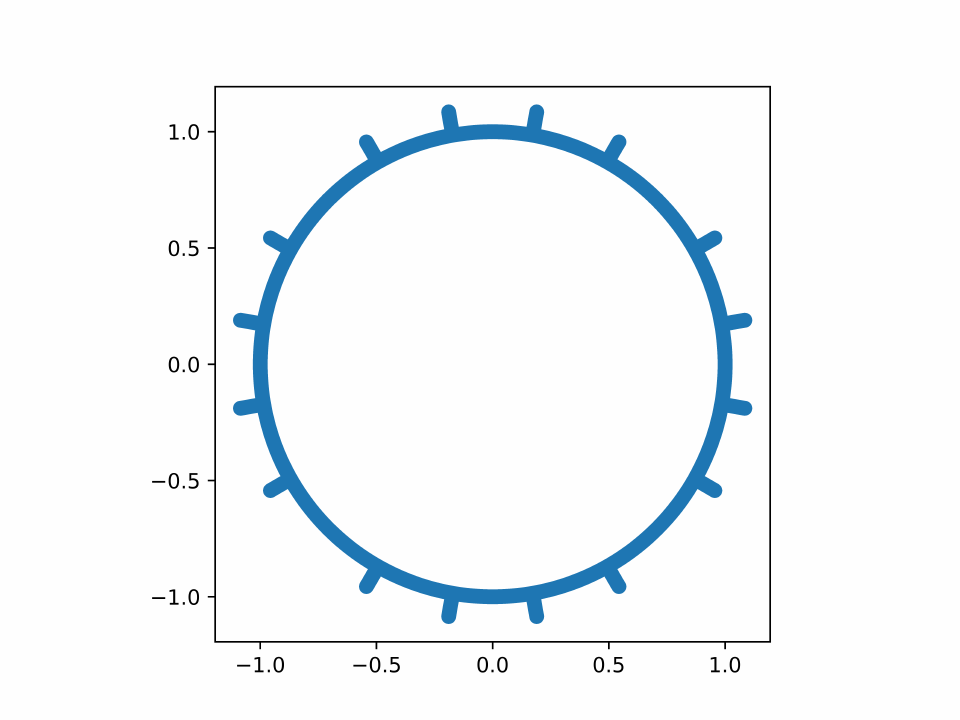}
				\caption{$\Phi(\E_2^{2,1})$}
				\label{fig-E21-image}
			\end{minipage}%
			\begin{minipage}{.5\textwidth}
				\centering
				\includegraphics[width=.9\linewidth]{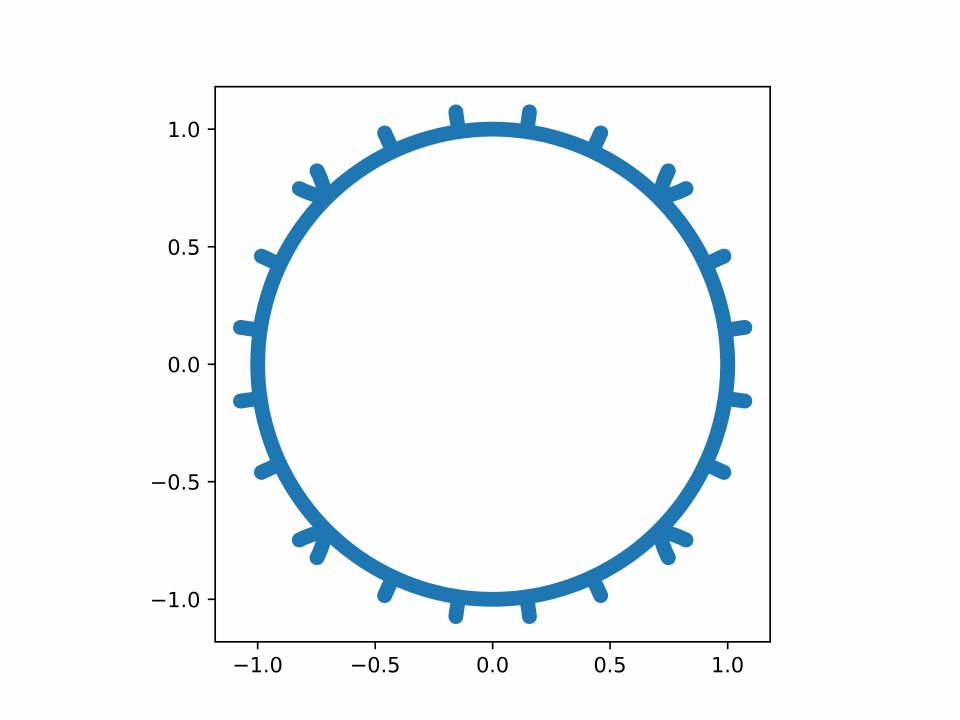}
				\caption{$\Phi(\E_2^{2,3})$}
				\label{fig-E23-image}
			\end{minipage}
		\end{figure}	
By construction, all these $4(2n+l-1)$ protrusions have the same equilibrium measure relative to $\Phi(\E_2^{n, l})$.  Hence their combined mass relative to equilibrium measure amounts to
		\begin{equation}  \label{12}
			\frac{2n+l-1}{4n+l}.
		\end{equation}	
As we shall argue below, the fact that this quantity is $<1/2$ for $l=1$ but $>1/2$ for $l=3$ seems to explain why $\Vert T_{2n+1}^{\E_2}\Vert_{\E_2}\lessgtr 2$ for $n$ even and odd, respectively.	
			
%Assuming that these protrusions were equidistributed on the circle and that each of them formed a straight line segments orthogonal to the circle, we would be dealing with a set of the form
Consider the sets
		\begin{equation}
		\label{domain}
			S(n,l):=\partial\bbD \, \cup\bigcup_{k=1}^{4(2n+l-1)}
			\exp\Bigl\{\frac{\pi ik/2}{2n+l-1}\Bigr\}%\exp\left\{\frac{\pi ik}{4n}\right\}\cdot
			 \,[1,R_{n, l}],
		\end{equation}
where $R_{n, l}>1$ is a suitable constant depending on $n$ and $l$.  These sets have a straightforward structure, allowing us to compute their capacity, and they also bear a resemblance to the configuration of $\Phi(\E_2^{n, l})$.  Although it may seem that way at first glance, the protrusions in $\Phi(\E_2^{n,l})$ are not perfectly straight line segments, and their distribution on $\partial\D$ is not entirely uniform.  One can show that by selecting $R_{n,l}$ in such a way that
		\begin{equation}
			\mu_{S(n, l)}(\partial\D) = \frac{2n+1}{4n+l},
		\end{equation}	 
the inequality
%		\begin{equation}
			$\Cap ( S(n, l) )^{4n+l} \lessgtr \sqrt{2}$
%		\end{equation}		
holds, depending on whether $l=1$ or $3$.  For example, by performing a straightforward calculation, we find that
	        \[
	           \Cap\bigl(S(n,1)\bigr) = {\Bigl(\cos\Bigl(\frac{\pi n}{4n+1}\Bigr)\Bigr)}^{-1/4n}
	        \]
and from this expression, one can infer that $\Cap (S(n,1) )^{4n+1} \to \sqrt{2}$ in a monotonically increasing manner.  Taking \eqref{eq:capacity_cheb_norm} into account, we now possess additional evidence indicating that 
%	\[
	   $\Vert T_{4n+1}^{\E_2}\Vert_{\E_2} \to 2$
%	\]    
from below.  Similar reasoning provides support for the conjecture that $\Vert T_{4n+3}^{\E_2}\Vert_{\E_2} \to 2$ from above.

\section{Outlook and perspective}
\label{sec:outlook}

A key finding of this paper has been to show that, beyond intervals, there exist compact subsets of $\C$ for which the Widom factors converge to $2$. Specifically, the sets $\E_m$ as defined in Theorem \ref{thm:widom-factors} and illustrated in Figures \ref{fig-E2}--\ref{fig-E11} exhibit this property when $m\geq 2$, and the same holds for any other symmetric star graph with an odd number of edges.
%A main result of the present article has been to show that there exist compact subsets of $\C$, besides intervals, for which the Widom factors converge to $2$.  The sets $\E_m$ defined in Theorem \ref{thm:widom-factors} and pictured in Figures \ref{fig-E2}--\ref{fig-E11} precisely have this property when $m\geq 2$ (and so does any other symmetric star graph with an odd number of edges). 
It is reasonable to inquire whether or not there are more sets $\E\subset\C$ for which $\cW_{n,\infty}(\E)\to 2$.  As explained below, we believe the answer is indeed affirmative.	

Recall that if $P$ is a complex polynomial and $P'(z)=0$, then $z$ is called a {critical point} and $w=P(z)$ a {critical value} for $P$.  Polynomials with at most two critical values, say $c_1$ and $c_2$, are known as \emph{Shabat polynomials} (or generalised Chebyshev polynomials), see \cite{BZ,SZ}.  Since this class of polynomials is invariant under non-degenerate linear transformations, we may assume that $c_1=-2$ and $c_2=2$.
It is a known fact (see, e.g., \cite{BZ} or \cite{JP}) that for Shabat polynomials, the preimage $P^{-1}([-2, 2])$ is a planar tree with as many edges as the degree of $P$.  In fact, one can establish a bijection between the set of bicolored planer trees and the set of equivalence classes of Shabat polynomials.  Moreover, with respect to the point at $\infty$, every edge of $P^{-1}([-2, 2])$ has equal harmonic measure and any subset of an edge has equal harmonic measure from both sides.  Trees with these properties are also called ``balanced trees'' and they satisfy the $S$-property of Definition \ref{def:s-property}.  A remarkable result of Bishop \cite{Bishop} shows that balanced trees are dense in all planar continua (i.e., compact connected subsets of $\C$) with respect to the Hausdorff metric.

For $n\geq 2$, the polynomial $P_n(z)=z^n$ has one critical point (namely, $z=0$) and one critical value $w=0$. So the preimage $P_n^{-1}([0,4])$ is a balanced tree and when $n=2m$, this tree coincides with $\E_m$ as defined in \eqref{Em}.  Another example of a Shabat polynomial (with critical values $w=0, 1$) is
\[
   P(z)=\frac{8}{729} (z+1) \Bigl(z^2-\frac{3}{2}z+\frac{9}{2}\Bigr)^3 %, \quad a=(34\pm 6\sqrt{21})/7
\] 
and the corresponding tree is pictured in Figure \ref{fig:shabat_tree_4}.  By use of the algorithm mentioned in Section \ref{sec:quadratic}, one can compute the norms of $T_n^{\E_P}$ for degree up to $40$ with a precision of $10^{-2}$.  The outcome is illustrated in Figure \ref{fig:shabat_norms_4}.

		\begin{figure}[h!]
			\centering
			\begin{minipage}{.5\textwidth}
		  		\centering
  				\includegraphics[width=0.8\linewidth]{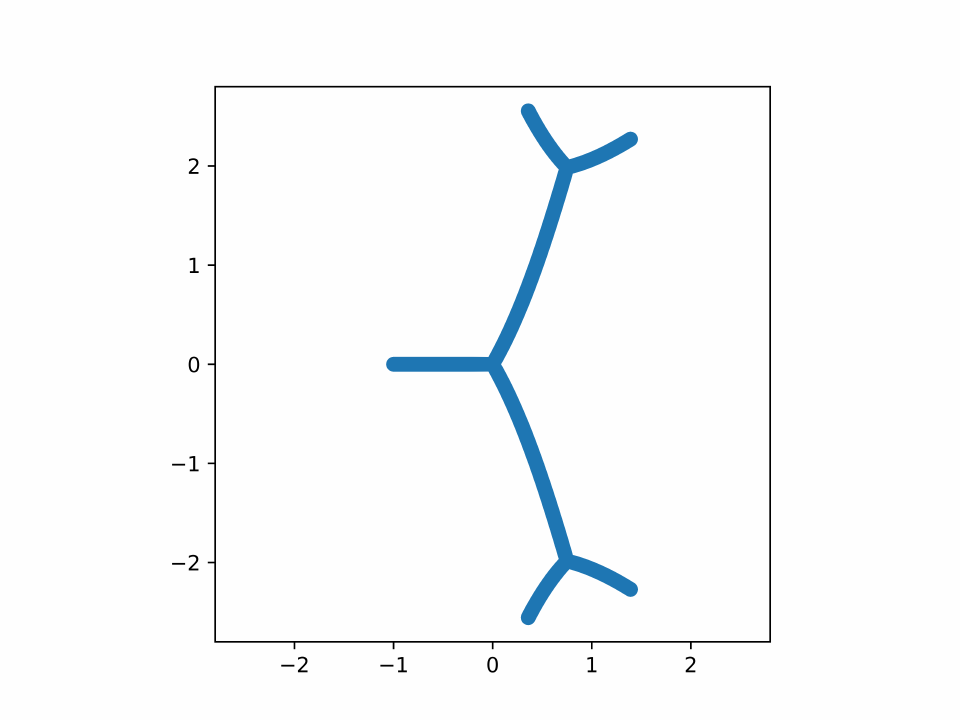}
				\caption{$\E_P=P^{-1}([0,1])$}
				\label{fig:shabat_tree_4}
			\end{minipage}%
			\begin{minipage}{.5\textwidth}
				\centering
				\includegraphics[width=.8\linewidth]{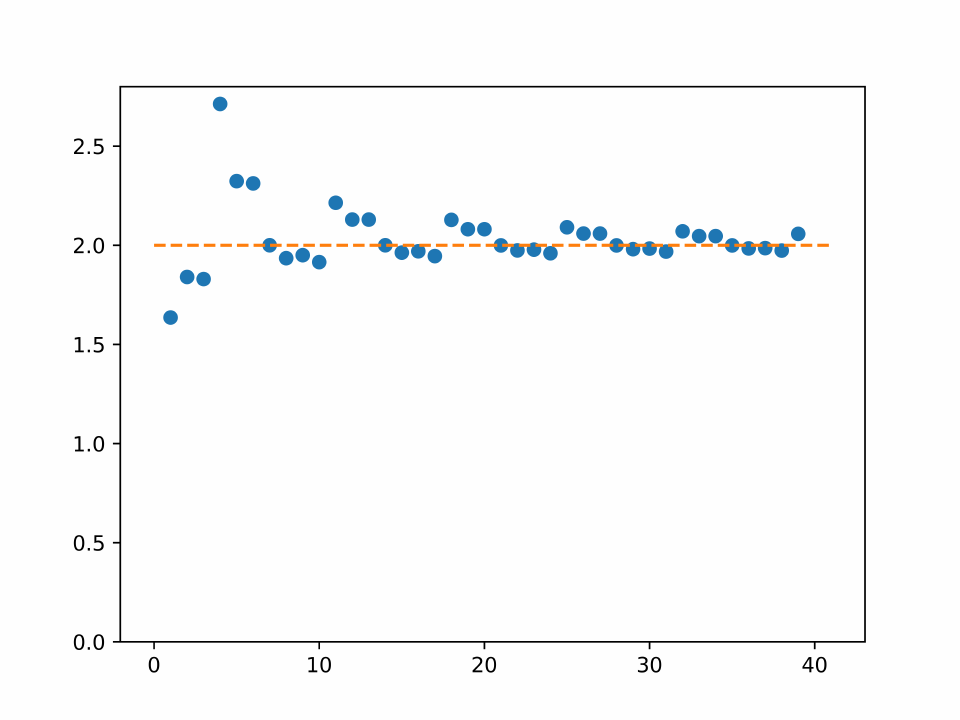}
				\caption{Widom factors for ${\E_P}$}
				\label{fig:shabat_norms_4}
			\end{minipage}
		\end{figure}

%\begin{figure}[h!]
%			\centering
%			\begin{minipage}{.4\textwidth}
%		  		\centering
%  				\includegraphics[width=1.2\linewidth]{shabat_tree}
%				\caption{$a=(34- 6\sqrt{21})/7$}
%				\label{fig:shabat_tree}
%			\end{minipage}%
%			\begin{minipage}{.6\textwidth}
%				\centering
%				\includegraphics[width=.8\linewidth]{shabat_norms}
%				\caption{$\|T_n^{\E}\|_{\E}$}
%				\label{fig:shabat_norms}
%			\end{minipage}
%		\end{figure}	

Similar plots for other Shabat polynomials lead to the same pattern and we thus have the courage to hypothesize that the first part of Theorem \ref{thm:widom-factors} applies to all balanced trees.

%\begin{figure}[h!]
%			\centering
%			\begin{minipage}{.4\textwidth}
%		  		\centering
%  				\includegraphics[width=1.2\linewidth]{shabat_tree_2}
%				\caption{$\E$ comes from $z^3(z-1)$}
%				\label{fig:shabat_tree_2}
%			\end{minipage}%
%			\begin{minipage}{.6\textwidth}
%				\centering
%				\includegraphics[width=.8\linewidth]{shabat_norms_2}
%				\caption{$\|T_n^{\E}\|_{\E}$}
%				\label{fig:shabat_norms_2}
%			\end{minipage}
%		\end{figure}	
%
%\begin{figure}[h!]
%			\centering
%			\begin{minipage}{.4\textwidth}
%		  		\centering
%  				\includegraphics[width=1.2\linewidth]{shabat_tree_3}
%				\caption{$\E$ comes from $z^7-\tfrac{21}{5}z^5+7z^3-7z+\tfrac{16}{5}$.}
%				\label{fig:shabat_tree_3}
%			\end{minipage}%
%			\begin{minipage}{.6\textwidth}
%				\centering
%				\includegraphics[width=.8\linewidth]{shabat_norms_3}
%				\caption{$\|T_n^{\E}\|_{\E}$}
%				\label{fig:shabat_norms_3}
%			\end{minipage}
%		\end{figure}	
		
\begin{conjecture}
\label{conj}
Let $P$ be a monic Shabat polynomial with critical values in $\{-2, 2\}$ and consider the preimage $\E_P:=P^{-1}([-2, 2])$. Then
\[
   \lim_{n\to\infty} \cW_{n, \infty} (\E_P)=2.
\]
\end{conjecture}

Since any balanced tree satisfies the $S$-property, one may speculate if this conjecture can be extended to a larger class of sets.  In fact, if the property of every edge having equal harmonic measure turns out to not play a role, it would be natural to formulate the conjecture for so-called generalised Shabat polynomials.  This class was introduced in \cite{JP} and allows for more than two critical values as long as they all belong to $[-2, 2]$.  The preimage $P^{-1}([-2, 2])$ is connected (in fact, a tree) if and only if $[-2, 2]$ contains all the critical values of $P$.  The ideal scenario would be that the conjecture holds true for all connected sets characterised by the $S$-property.

Regrettably, our current method of proof does not exhibit a clear path for generalisation within the context of Shabat polynomials and the $S$-property.  Novel strategies and ideas are required to advance our understanding in this domain.
	
	\appendix
    \section{Bernstein's method}
	
    In this appendix we shall discuss a proof of Theorem \ref{thm:Bernstein}.  Our approach will be a simplification of the method employed by Bernstein in \cite{Bernstein:1930-31}.  The initial stage of \cite{Bernstein:1930-31} involved showing that if $w$ is a polynomial which is strictly positive on $[-1, 1]$, then the monic orthogonal polynomials $P_n^{w}$ with respect to the weight function
    \[
       {w(x)}/{\sqrt{1-x^2}}
    \]
are well-suited trial polynomials for minimisation of the quantity
    %\begin{equation}
    %\label{max sqrt}
    \[
        \max_{x\in [-1,1]}\left|\sqrt{w(x)}\left(x^n+\sum_{k=0}^{n-1}c_kx^k\right)\right|.
    \]
    %\end{equation}
As in Section \ref{sec:quadratic}, we shall denote these minimisers by $T_n^{\sqrt{w}}$.  More specifically, Bernstein proved that $\sqrt{w}P_n^{w}$ is asymptotically alternating for this class of polynomial weights and an explicit analysis of $P_n^w$ can then be used to unveil the asymptotic behaviour of $\|\sqrt{w}T_n^{\sqrt{w}}\|_{[-1,1]}$ as stated in Theorem \ref{thm:Bernstein}.
    
    Rather than using orthogonal polynomials as trial polynomials, we shall base our analysis on an explicit expression for certain weighted Chebyshev polynomials due to Achieser \cite[Appendix A]{Ach_2}.  This formula is presented in Section \ref{A1}.  Since the class of weights includes the reciprocal of any strictly positive polynomial on $[-1, 1]$, we can --- in line with Bernstein --- apply a standard approximation argument to conclude that Theorem \ref{thm:Bernstein} holds true for all Riemann integrable weights on $[-1, 1]$ which are bounded above and below by positive constants.  This is the second step in the proof and we outline the details in Section \ref{A2}.    
    %We will consider a different approach of proving this result which shortens the argument significantly. Instead of proving that certain orthogonal polynomials are asymptotically alternating we will instead consider a result due to Achieser which gives the explicit Chebyshev polynomials for a certain class of weights. 
%    We first show that for weights which are given as the reciprocal of a zero free square root of a polynomial the corresponding weighted Chebyshev polynomials can be explicitly represented for sufficiently large degrees.
%Using this result we then obtain the complete norm asymptotics for weighted Chebyshev polynomials on $[-1,1]$ corresponding to Riemann integrable weights which are bounded from below and above. This step of the proof uses that any such weight can be approximated by zero free polynomials.
The final and most complicated step is to allow for zeros of the weight function on $[-1, 1]$. We shall present exactly the same refined argument as provided by Bernstein \cite{Bernstein:1930-31} in Section \ref{A3}.   
    
    \subsection{An explicit formula for certain weighted Chebyshev polynomials}
    \label{A1}
    Denote by $\overline{\R}:=\R\cup\{\infty\}$ the extended real line.
    Let $a_k\in\overline{\R}\setminus [-1,1]$ and form the product
       \[
          w(x) = \prod_{k=1}^{2m}\left(1-\frac{x}{a_k}\right)^{-1/2},
       \]
    which represents a positive weight function on $[-1,1]$. The case of an odd number of factors can be handled by taking $a_{2m} = \infty$ and $\vert a_k \vert<\infty$ for $k=1, 2, \ldots, 2m-1$. 
    %As before we let $T_n^{w}$ denote the $n$th degree monic minimiser of 
    %\[\max_{x\in [-1,1]}\left|w(x)\left(x^n+\sum_{k=0}^{n-1}c_kx^k\right)\right|\]
    %whose existence and uniqueness in this case is classical. 
Let $z\in\partial\D$ and $\rho_k\in\D$ be defined implicitly by
    \[
       x = \frac{1}{2}\left(z+\frac{1}{z}\right) \; \mbox{ and } \; 
       a_k  = \frac{1}{2}\left(\rho_k+\frac{1}{\rho_k}\right) \; \mbox{ for } \; k=1, 2, \ldots, 2m.
    \]
The following explicit representation for $wT_n^{w}$ is given by Achieser \cite[Appendix A]{Ach_2}.
    \begin{theorem}
        \label{thm:akhiezer}
        For $n>m$, we have
        \begin{equation}
        \label{wT}
           w(x)T_n^{w}(x) = 2^{-n}\prod_{k=1}^{2m}\sqrt{1+\rho_k^2}
           \left(z^{m-n} \prod_{k=1}^{2m}\sqrt{\frac{1-z\rho_k}{z-\rho_k}}+
           z^{n-m}\prod_{k=1}^{2m}\sqrt{\frac{z-\rho_k}{1-z\rho_k}}\,\right)
        \end{equation}
        and
        \begin{equation}
            \|wT_n^{w}\|_{[-1,1]} = 2^{1-n}\exp\left\{\frac{1}{\pi}\int_{-1}^{1}\frac{\log w(x)}{\sqrt{1-x^2}}dx\right\}.
            \label{eq:achieser_norm}
        \end{equation}
    \end{theorem}
\begin{remark}
One can show that the expression on the right-hand side of \eqref{wT} indeed becomes a polynomial (in $x$) after division by $w(x)$.  Moreover, this polynomial satisfies the orthogonality conditions
\[
   \int_{-1}^1 T_n^w(x) x^k \frac{w(x)^2}{\sqrt{1-x^2}} dx = 0, 
   \quad k=0, 1, \ldots, n-1
\] 
when $n\geq m$. 
\end{remark}

By letting $a_{2k}=a_{2k-1}$ for $k=1, 2, \ldots, m$, we see that Theorem \ref{thm:akhiezer} in particular applies to weights of the form
    \[
       w(x) = \prod_{k=1}^{m}\left(1-\frac{x}{a_k}\right)^{-1}.
%       \quad b_k\in\overline{\R}\setminus [-1,1].
    \]
%with $a_k\in\overline{\R}\setminus [-1,1]$.  
This was in fact already proven by Chebyshev \cite{Chebyshev1859}.

    \subsection{Norm asymptotics for more general weighted Chebyshev polynomials}
    \label{A2}
    
	We now consider generalising \eqref{eq:achieser_norm} to the case where $w$ is merely a Riemann integrable function on $[-1,1]$ satisfying
	\begin{equation}
	    \label{eq:weight_bounds}
	    \frac{1}{M} \leq w(x) \leq M       
	\end{equation}
for some $M\geq 1$.  In this more general setting, we no longer have equality. But the two expressions in \eqref{eq:achieser_norm} are still asymptotically equivalent as $n\to\infty$.

   The idea is to approximate $w$ by reciprocals of polynomials.  Set $\omega = 1/w$
%Then $\omega$ is Riemann integrable and further satisfies \eqref{eq:weight_bounds}. 
and choose sequences $\{\omega_u^k\}_k$ and $\{\omega_l^k\}_k$ of polynomials with real zeros away from $[-1,1]$ such that
    \[
       \frac{1}{2M}\leq \omega_u^k(x)\leq \omega(x)\leq \omega_l^k(x)\leq 2M
    \]
and 
    \[
		\lim_{k\rightarrow \infty} \omega_u^{k}(x) = \omega(x) = \lim_{k\rightarrow \infty}\omega_l^{k}(x)
	\]
for every $x\in [-1,1]$.  Defining $w_l^k = 1/\omega_l^k$ and $w_u^k = 1/\omega_u^k$, we obtain that 
    \[
       w_l^k(x)\leq w(x)\leq w_u^k(x)
    \] 
while maintaining the pointwise convergence
    \[
       \lim_{k\rightarrow \infty} w_l^{k}(x) = w(x) = \lim_{k\rightarrow \infty}w_u^{k}(x).
    \]
This in turn implies that
	\begin{equation}  \label{eq:monotonicity_cheb_weight}
	    \|T_n^{w_l}w_l\|_{[-1,1]}\leq \|T_n^{w}w_l\|_{[-1,1]}
	    \leq \|T_n^{w}w\|_{[-1,1]}\leq \|T_n^{w_u}w\|_{[-1,1]}
	    \leq \|T_n^{w_u}w_u\|_{[-1,1]}.    
	\end{equation}
Furthermore, using dominated convergence we get that
	\[
		\lim_{k\rightarrow \infty}\int_{-1}^{1}\frac{\log w_l^k(x)}{\sqrt{1-x^2}}dx = \int_{-1}^{1}\frac{\log w(x)}{\sqrt{1-x^2}}dx = \lim_{k\rightarrow \infty}\int_{-1}^{1}\frac{\log w_u^k(x)}{\sqrt{1-x^2}}dx
	\]
and for any given $\epsilon>0$, we can therefore choose $k$ large enough that 
	\begin{align*}\left|\exp\left\{\frac{1}{\pi}\int_{-1}^{1}\frac{\log w_i^k(x)}{\sqrt{1-x^2}}dx\right\} - \exp\left\{\frac{1}{\pi}\int_{-1}^{1}\frac{\log w(x)}{\sqrt{1-x^2}}dx\right\}\right|&<\epsilon,
	\quad i=l, u.
	\end{align*}
 %   with $i = l,u$.
For $n$ sufficiently large, Theorem \ref{thm:akhiezer} implies that 
	\begin{align*}
		\|T_n^{w_i^k}w_i^k\|_{[-1,1]}=
		2^{1-n}\exp\left\{\frac{1}{\pi}\int_{-1}^{1}\frac{\log w_i^k(x)}{\sqrt{1-x^2}}dx\right\},
		\quad i = l,u.
	\end{align*}
By combining this with \eqref{eq:monotonicity_cheb_weight}, we thus get the inequalities
	\begin{align*}
		\exp\left\{\frac{1}{\pi}\int_{-1}^{1}\frac{\log w(x)}{\sqrt{1-x^2}}dx\right\}-\epsilon\leq 2^{n-1}\|T_n^ww\|_{[-1,1]}\leq \exp\left\{\frac{1}{\pi}\int_{-1}^{1}\frac{\log w(x)}{\sqrt{1-x^2}}dx\right\}+\epsilon.
	\end{align*}
As $\epsilon>0$ is arbitrary, we conclude that
	\begin{equation}
	\|T_n^w w\|_{[-1,1]} = 
	2^{1-n}\exp\left\{\frac{1}{\pi}\int_{-1}^{1}\frac{\log w(x)}{\sqrt{1-x^2}}dx\right\}
	\bigl(1+o(1)\bigr)
	\label{eq:non_zero_weighted_cheb_asymptotics}
	\end{equation}
as $n\rightarrow \infty$.  This proves Theorem \ref{thm:Bernstein} when $w$ is Riemann integrable and satisfies \eqref{eq:weight_bounds}.
	
	\subsection{Allowing for zeros of the weight function}
	\label{A3}
	
	The final step in proving Theorem \ref{thm:Bernstein} is to allow for zeros at certain points of $[-1,1]$.  We will adopt the approach outlined in \cite{Bernstein:1930-31} and restrict our examination to the introduction of a single zero $b_1 = b$ within the interval $[-1, 1]$ using the weight $|x-b|^{\alpha}$.  That is, we consider weights of the form
	\begin{equation}  \label{w}
	   w(x)=w_0(x) \vert x-b \vert^\a,
	\end{equation}	
where $w_0$ is Riemann integrable and satisfies $1/M \leq w_0(x) \leq M$ for some constant $M\geq 1$.  More zeros can be added by repeated use of the argument we are about to explain.  

   Assume first that $\alpha=m$ is a positive integer. For $\eps>0$, consider the weight function
		\[
			w^\eps(x) = w_0(x)|x-b-i\eps|^m.
		\]
		This weight is non-vanishing on $[-1,1]$ and fulfils all the previous requirements for \eqref{eq:non_zero_weighted_cheb_asymptotics} to be valid.  By use of Lemma \ref{lem:log_integral_equilibrium_measure}, we hence find that
	\begin{align}
         \|T_n^{w^\eps} w^\eps\|_{[-1, 1]} \notag
         %2^{1-n}\exp\left\{\frac{1}{\pi}\int_{-1}^{1}\frac{\log w^\delta(x)}{\sqrt{1-x^2}}dx\right\}(1+o(1)) \\
	& = 2^{1-n} \exp\left\{\frac{1}{\pi}\int_{-1}^{1}\frac{\log w_0(x)}{\sqrt{1-x^2}}dx\right\}
	 \exp\left\{\frac{1}{\pi}\int_{-1}^{1}\frac{\log |x-b-i\eps|^{m}}{\sqrt{1-x^2}}dx\right\}
	 \bigl(1+o(1)\bigr) \\  
	& = \|T_n^{w_0}w_0\|_{[-1,1]}
	\Biggl(\frac{\left|(b+i\eps)+\sqrt{(b+i\eps)^2-1}\right|}{2}\Biggr)^m
	\bigl(1+o(1)\bigr)   
            \label{eq:asymptotics_upper}
		\end{align}
as $n\to\infty$.  Note that the second factor in \eqref{eq:asymptotics_upper}
       % \[\left(\frac{\left|(b+i\delta)+\sqrt{(b+i\delta)^2-1}\right|}{2}\right)^m\]
        converges to $2^{-m}$ as $\eps\rightarrow 0$.  Since $w^\eps>w$ on $[-1,1]$, we deduce as in \eqref{eq:monotonicity_cheb_weight} that 
		\begin{align}
            \begin{split}
            \|T_n^{w^\eps} w^\eps\|_{[-1,1]}& \geq \|T_n^{w^\eps}w\|_{[-1,1]}
            \geq \|T_n^ww\|_{[-1,1]}\\&\geq \|T_{n+m}^{w_0}w_0\|_{[-1,1]} 
            = 2^{-m}\|T_{n}^{w_0}w_0\|_{[-1,1]}\bigl(1+o(1)\bigr).    
            \end{split}
   \label{eq:upper_lower_bound}
		\end{align}
The combination of \eqref{eq:asymptotics_upper} and \eqref{eq:upper_lower_bound} now implies that
	\[
		\|T_n^ww\|_{[-1,1]} = 2^{-m}\|T_n^{w_0}w_0\|_{[-1,1]}\bigl(1+o(1)\bigr) 
		= 2^{1-n}\exp\left\{\frac{1}{\pi}\int_{-1}^{1}\frac{\log w(x)}{\sqrt{1-x^2}}dx\right\}
		   \bigl(1+o(1)\bigr)
		\]
as $n\to\infty$, and this proves that \eqref{eq:non_zero_weighted_cheb_asymptotics} is valid for weights of the form \eqref{w} with $\a=m\in\N$. 

We next consider the case of negative integer powers.  If $\a=-m$ is a negative integer, then $T_{n+m}^{w}$ has a zero of order $m$ at $x = b$ and hence
		\[
			T_{n+m}^{w}(x)w(x) = T_n^{w_0}(x)w_0(x)\frac{(x-b)^m}{|x-b|^m}.
		\]
Therefore,
		\begin{align*}
		  \|T_{n+m}^{w}w\|_{[-1,1]} &= 
		  2^{1-n}\exp\left\{\frac{1}{\pi}\int_{-1}^{1}\frac{\log w_0(x)}{\sqrt{1-x^2}}dx\right\}
		  \bigl(1+o(1)\bigr) \\
		  & = 2^{1-n-m}\exp\left\{\frac{1}{\pi}\int_{-1}^{1}\frac{\log w(x)}{\sqrt{1-x^2}}dx\right\}
		  \bigl(1+o(1)\bigr)
		\end{align*}
as $n\to\infty$, and this proves the validity of \eqref{eq:non_zero_weighted_cheb_asymptotics} for such weights as well. 
		
		%To extend these asymptotic formulae to include the case of weights whose representation is given by $w_0(x)|x-b|^\alpha$ 
		To handle arbitrary exponents $\alpha\in \bbR$, it suffices to examine the case of $\alpha\in (0,1)$. The weight function $w_0$ can namely incorporate zeros with integer exponents on $[-1,1]$ as we have already addressed the asymptotics in this particular case. Let $\eps>0$ and form the weight functions
		\begin{align*}
			w_l^{\eps}(x)&=\begin{cases}
				w(x),& |x-b|\geq\eps \\
				w_0(x)|x-b|,& |x-b|<\eps,
			\end{cases}  \\
			w_u^{\eps}(x)&=\begin{cases}
				w(x),& |x-b|\geq\eps \\
				w_0(x),& |x-b|<\eps.
			\end{cases}
		\end{align*}
Note that \eqref{eq:non_zero_weighted_cheb_asymptotics} applies to both $w_l^{\eps}$ and $w_u^{\eps}$. Regardless of the value of $0<\eps<1$, we have the inequalities
		\[
			w_l^{\eps}(x)\leq w(x)\leq w_u^{\eps}(x)
		\]
and hence %\eqref{eq:monotonicity_cheb_weight} implies that
		\begin{equation*}
		   \|T_n^{w_l^{\eps}}w_l^{\eps}\|_{[-1,1]}
		   \leq \|T_n^{w}w\|_{[-1,1]}
		   \leq \|T_n^{w_u^{\eps}}w_u^{\eps}\|_{[-1,1]}.\label{eq:weighted_chebyshev_inequalities}
		   \end{equation*}	
Consequently,
	 \begin{multline*}
	    \qquad 2^{1-n}\exp\left\{\frac{1}{\pi}\int_{-1}^{1}
	    \frac{\log w_l^{\eps}(x)}{\sqrt{1-x^2}}dx\right\}
	    \bigl(1+o(1)\bigr) \\
	    \leq \|T_n^ww\|_{[-1,1]}  
	    \leq 2^{1-n}\exp\left\{\frac{1}{\pi}\int_{-1}^{1}
	    \frac{\log w_u^{\eps}(x)}{\sqrt{1-x^2}}dx\right\} 
	    \bigl(1+o(1)\bigr)  \qquad
	 \end{multline*}
as $n\rightarrow \infty$. By dominated convergence, 
	\[
		\lim_{\eps\rightarrow 0} \int_{-1}^{1}\frac{\log w_l^{\eps}(x)}{\sqrt{1-x^2}}dx = 
		\int_{-1}^{1}\frac{\log w(x)}{\sqrt{1-x^2}}dx =
		\lim_{\eps\rightarrow 0} \int_{-1}^{1}\frac{\log w_u^{\eps}(x)}{\sqrt{1-x^2}}dx,
	\]
%	 	Equation \eqref{eq:weighted_chebyshev_inequalities} together with the asymptotic formulae for $w_l^\delta$ and $w_u^\delta$ gives us that for a fixed $\delta>0$
so for any given $r>0$, we can choose $\eps>0$ small enough that
	 \[
	     1-r \leq \frac{\exp\left\{\frac{1}{\pi}\int_{-1}^{1}
	     \frac{\log w_l^{\eps}(x)}{\sqrt{1-x^2}}dx\right\}}
	     {\exp\left\{\frac{1}{\pi}\int_{-1}^{1}\frac{\log w(x)}{\sqrt{1-x^2}}dx\right\}}
	     \leq 1,  \qquad 
	     1\leq \frac{\exp\left\{\frac{1}{\pi}\int_{-1}^{1}
	     \frac{\log w_u^{\eps}(x)}{\sqrt{1-x^2}}dx\right\}}
	     {\exp\left\{\frac{1}{\pi}\int_{-1}^{1}\frac{\log w(x)}{\sqrt{1-x^2}}dx\right\}}
	     \leq 1+r.
	 \]
This in turn implies that
	\begin{multline*}
	    \qquad 2^{1-n}\exp\left\{\frac{1}{\pi}\int_{-1}^{1}
	    \frac{\log w(x)}{\sqrt{1-x^2}}dx\right\} (1-r) \bigl(1+o(1)\bigr) \\
	    \leq \|T_n^w w\|_{[-1,1]}
	    \leq 2^{1-n}\exp\left\{\frac{1}{\pi}\int_{-1}^{1}
	    \frac{\log w(x)}{\sqrt{1-x^2}}dx\right\}(1+r)\bigl(1+o(1)\bigr)  \qquad
	 \end{multline*}
as $n\rightarrow \infty$. Since $r>0$ is arbitrary, we conclude that
	 \[
	    \|T_n^ww\|_{[-1,1]} = 
	    2^{1-n}\exp\left\{\frac{1}{\pi}\int_{-1}^{1}
	    \frac{\log w(x)}{\sqrt{1-x^2}}dx\right\}\bigl(1+o(1)\bigr)
	 \]
as $n\rightarrow \infty$. In other words, \eqref{eq:non_zero_weighted_cheb_asymptotics} --- which coincides with \eqref{B asymp} --- applies.  The addition of more weights of the form $|x-b_k|^{\alpha_k}$ can be carried out by repeated use of the argument explained above.
 	
%	\cite{latex2e}
%	\bibliographystyle{plain} % We choose the "plain" reference style
%	\bibliography{refs}

%\end{document}
	
%%%%%%%%%%%%%%%%%%%%%%%%%%%%%%%

\end{document}